\documentclass[journal]{IEEEtran}

\usepackage{cite}
\ifCLASSINFOpdf
\usepackage[pdftex]{graphicx}
\graphicspath{{../pdf/}{../jpeg/}}
\DeclareGraphicsExtensions{.pdf,.jpeg,.png}
\else
\fi

\usepackage[cmex10]{amsmath}

\usepackage{amssymb, bm}

\DeclareMathOperator*{\argmin}{arg\,min}
\usepackage{comment}
\usepackage{array,multirow}
\usepackage{tikz}

\usepackage[utf8]{inputenc} % allow utf-8 input
\usepackage[T1]{fontenc}    % use 8-bit T1 fonts
\usepackage{hyperref}       % hyperlinks

\usepackage[noabbrev,capitalise]{cleveref}
\usepackage{accents}

\usepackage{bbm}
\usepackage{dsfont}
\usepackage{bbold}

\usepackage{algorithm}
\usepackage{algorithmic}
\usepackage{enumerate,enumitem}
\usepackage{amsthm}
\theoremstyle{remark}
\newtheorem{theorem}{Theorem}
\newtheorem{lemma}{Lemma}

\newtheorem{definition}{Definition}
\newtheorem{proposition}{Propostion}

\usepackage{dsfont}

\usepackage{mathtools}

\usepackage{color,soul}

\renewcommand{\baselinestretch}{0.982}

\begin{document}
\title{Robust AC Optimal Power Flow \\ with Robust Convex Restriction}
\author{Dongchan Lee, Konstantin Turitsyn, Daniel K. Molzahn, and Line A. Roald
\thanks{
This work was supported in part by the U.S. Department of Energy Office of Electricity as part of the DOE Grid Modernization Initiative and in part by the National Science Foundation Energy, Power, Control and Networks Award 1809314.

D. Lee is with the Department of Mechanical Engineering, Massachusetts Institute of Technology, Cambridge, MA 02139, USA (email: dclee@mit.edu).

K. Turitsyn is with the D. E. Shaw Group, New York, NY 10036 USA (e-mail: turitsyn@mit.edu).

D. K. Molzahn is with the School of Electrical and Computer Engineering, Georgia Institute of Technology, Atlanta, GA 30313 USA (e-mail: molzahn@gatech.edu).

L. A. Roald is with the Department of Electrical and Computer Engineering, University of Wisconsin, Madison, WI 53706 USA (e-mail: roald@wisc.edu).
}
}


\maketitle

\begin{abstract}
Electric power grids regularly experience uncertain fluctuations from load demands and renewables, which poses a risk of violating operational limits designed to safeguard the system.
In this paper, we consider the robust AC OPF problem that minimizes the generation cost while requiring a certain level of system security in the presence of uncertainty.
The robust AC OPF problem requires that the system satisfy operational limits for all uncertainty realizations within a specified uncertainty set. 
Guaranteeing robustness is particularly challenging due to the non-convex, nonlinear AC power flow equations, which may not always have a solution. 
In this work, we extend a previously developed convex restriction to a \emph{robust convex restriction}, which is a convex inner approximation of the non-convex feasible region of the AC OPF problem that accounts for uncertainty in the power injections.
We then use the robust convex restriction in an algorithm that obtains robust solutions to AC OPF problems by solving a sequence of convex optimization problems. We demonstrate our algorithm and its ability to control robustness versus operating cost trade-offs using PGLib test cases.
\end{abstract}

\IEEEpeerreviewmaketitle

\section{Introduction}
The optimal power flow (OPF) problem determines the minimum cost dispatch that satisfies the power flow equations and operational constraints such as voltage magnitude, line flow, and generator limits~\cite{ferc4}. 
OPF problems use the AC power flow equations to model the flow of power and balance supply and demand forecasts, which enter as parameters in the optimization problem. Increasing penetrations of stochastic renewable generation, often observed as large variations in the demand due to behind-the-meter solar PV, lead to the actual operating conditions differing from the forecast considered in the OPF problem.
Forecast errors can result in unacceptable violations of operational limits and the potential loss of steady-state stability (i.e., a condition where the system attempts to reach an operating point for which the AC power flow equations do not admit any solutions). An AC representation of the power flow equations is needed to model these effects accurately. 

\subsection{Literature Review}
\subsubsection{Robust AC Optimal Power Flow} Common approaches for considering uncertainties in OPF problems include \emph{stochastic} \cite{Phan2014}, \emph{chance-constrained} \cite{roald2013, bienstock2014chance, venzke2017, roald2018cc, muhlpfordt2019chance, marley2016}, and \emph{robust} \cite{Jabr, capitanescu2012,molzahn_roald-acopf_robust2018,molzahn2019hicss, Bai2016, lorca2017robust, amjady2017, louca2018, Chamanbaz} formulations. 
Robust OPF formulations, which are considered in this paper, prohibit constraint violations for all uncertainty realizations within a deterministic uncertainty set~\cite{BenTal2009}. 
Fig.~\ref{fig:intro} illustrates the difference between a nominal OPF solution and a robust OPF solution for a 9-bus system. The feasible region (in blue) has two ``holes'' induced by two voltage magnitude constraints. We observe that the nominal solution (in purple) is at the voltage limit, making it vulnerable to changes in the parameters.
The robust solution (red dot) requires that any perturbation inside a \emph{confidence ellipsoid} remains feasible with respect to the operational constraints.
The confidence ellipsoid represents a set of plausible perturbations that captures the uncertain variable with a certain probability (e.g., with 95\% confidence)~\cite{BenTal2009}. This definition establishes a natural link to chance-constrained formulations, which require that constraint violation probabilities are less than a specified threshold.

\begin{figure}[!t]
	\centering
	\includegraphics[width=3in]{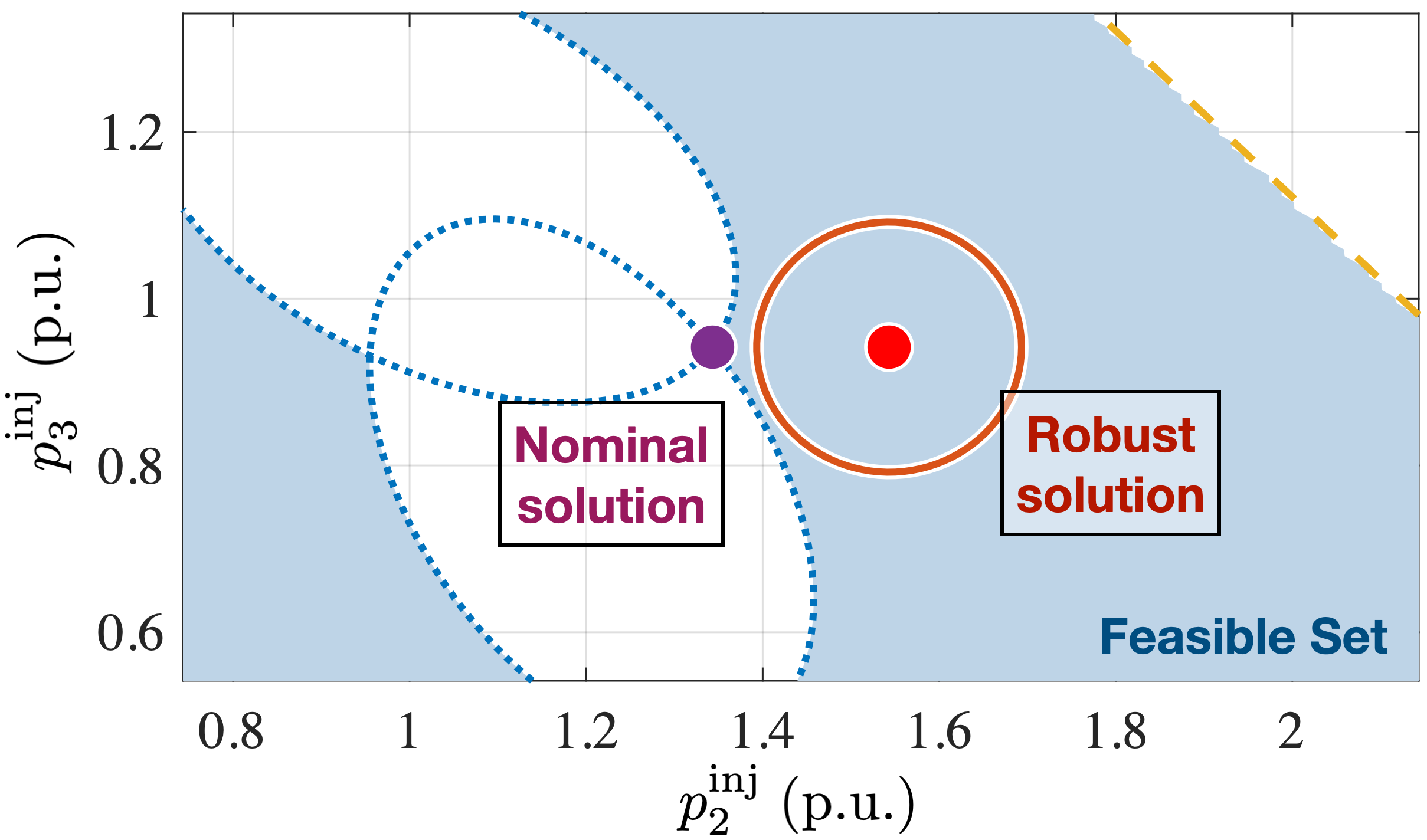}
	\caption{Illustration of the nominal (non-robust) solution and a robust OPF solution for a 9-bus system. The blue region is the nonconvex feasible space defined by the maximum voltage magnitude limits at buses 6 and 8 (two blue ovals) and minimum active power generation limit at the slack bus (yellow dashed line at the top right corner). Whereas generation uncertainty at buses~2 and 3~can result in the nominal solution (purple dot) easily violating the voltage limits, the robust solution (red dot) withstands all uncertainty realizations that are within the confidence ellipsoid (red circle).}
	\label{fig:intro}
	\vspace*{-1em}
\end{figure}

\begin{figure*}[!t]
	\centering
	\includegraphics[width=6.5in]{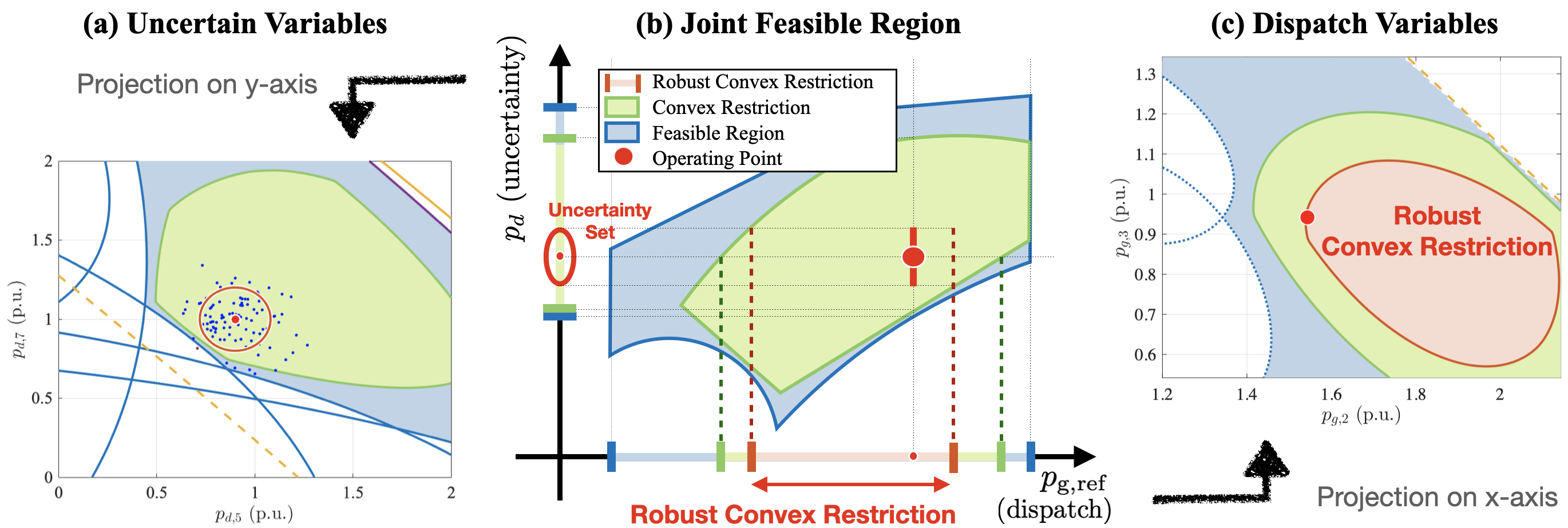}
	\caption{Illustration of a convex restriction and a robust convex restriction showing projections of (a)~uncertain variables on both axes, (b)~an uncertain variable on the $y$-axis and a dispatch variable on the $x$-axis, and (c)~dispatch variables on both axes. The red dots on the figures correspond to the nominal generation and load, and the circles in Figs.~\ref{fig:robust_cvxrs}(a) and \ref{fig:robust_cvxrs}(b) show the range of uncertainty at the load buses. The convex restriction from \cite{lee2018} (in green) is a convex subset of the feasible region of the non-convex AC OPF problem (in blue).
	The robust convex restriction derived in this paper takes the range of load scenarios described by the uncertainty set (red circle on the $y$-axis of Fig.~\ref{fig:robust_cvxrs}(b)) as an input and determines the range of dispatch variables (red range on the $x$-axis of Fig.~\ref{fig:robust_cvxrs}(b)) that ensures feasibility for all load scenarios in the uncertainty set. Figs.~\ref{fig:robust_cvxrs}(a) and \ref{fig:robust_cvxrs}(c) correspond to the IEEE 9-bus system, while 
	Fig.~\ref{fig:robust_cvxrs}(b) is a fictitious illustrative example.
	}
	\label{fig:robust_cvxrs}
\end{figure*}

Guaranteeing the robustness of AC OPF solutions against power injection uncertainty requires two conditions: for all uncertainty realizations, (i)~there exists a solution to the AC power flow equality constraints, and (ii)~the solution satisfies operational limits formulated as inequality constraints.\footnote{While protection against component failures (for instance, by enforcing N-1 security constraints~\cite{stott2012,Capitanescu}) is another essential consideration in many practical applications, this paper focuses on uncertain power injections.
}
Guaranteeing a solution to the AC power flow equations is particularly challenging since it is hard to capture the implicit nonlinear relationship between decision variables. Satisfying the AC power flow equations is important because this is a necessary condition for steady-state feasibility. However, most of the algorithms developed for chance-constrained and robust OPF problems simplify the AC power flow equations~\cite{molzahn_hiskens-fnt2019}, particularly through the use of linearizations~\cite{Jabr,roald2013,bienstock2014chance, roald2017irep, roald2018cc, Xie18} and convex relaxations~\cite{venzke2017}.
Other approaches in~\cite{molzahn_roald-acopf_robust2018,molzahn2019hicss,Bai2016} use convex relaxations to verify operational constraint satisfaction for a given uncertainty set, but do not guarantee AC power flow solvability for all uncertainty realizations.
The approach in~\cite{louca2018} relaxes the AC power flow equations by assuming that all buses have adjustable generators to balance power flows, which is not the case in many practical settings. 
Reference~\cite{muhlpfordt2019chance} uses polynomial chaos expansion to approximate the impacts of the uncertainties, with increasing accuracy coming at the expense of computational tractability.
Other related work in~\cite{amjady2017,lorca2017robust} considers multi-period OPF problems in the presence of uncertainty by iteratively solving the AC OPF problems with linearizations or convex relaxations. The OPF problem in~\cite{Nick18} considers the nonlinear AC power flow equations, but the model is limited to radial networks.
Alternative methods in~\cite{marley2016,Phan2014,Chamanbaz,capitanescu2012} use scenario-based approaches that enforce feasibility for selected uncertainty realizations. 
Despite much interest, existing algorithms for robust AC OPF do not guarantee the solution's feasibility to the full, nonlinear AC power flow equations for general, meshed network configurations.

\subsubsection{Convex Restriction}
Prior work by the authors in~\cite{nguyen2018constructing, lee2018, lee2019feasible} develops so-called \emph{convex restrictions} that provide a foundation for solving robust AC OPF problems with rigorous guarantees regarding both the solvability of AC power flow equations and the satisfaction of operational limits. The convex restrictions developed and applied in~\cite{nguyen2018constructing, lee2018, lee2019feasible} are conservative convex inner approximations of the feasible spaces of AC OPF problems. In other words, each point within the region defined by the convex restriction has an AC power flow solution that satisfies the operational limits. This definition contrasts with \emph{convex relaxations}, which extend the nonconvex AC~OPF feasible space to become convex by adding infeasible points. Reference~\cite{nguyen2018constructing} proposed the initial theory for convex restrictions with respect to the AC power flow equations and certain operational constraints. Reference~\cite{lee2018} developed new techniques for addressing the nonlinear terms in the AC power flow equations along with other extensions that significantly expand the capabilities of the convex restrictions. In~\cite{lee2019feasible}, the convex restrictions were applied to compute a path between operating points such that all points on the path are guaranteed to be feasible. 
% that gives rise to feasible solutions.

Techniques discussed in this work can be applied to settings other than AC power flow equations. However, similar to convex relaxations, convex restrictions should always be tailored to the specific nonlinearities of the problem.  For example, in~\cite{lee2020constrained}, convex restrictions are adapted to nonlinear time-difference equations arising in a different context, namely robust model predictive control for autonomous navigation.

%\begin{figure}[!t] \centering
%	\includegraphics[width=3.2in]{figures/illustration2.png}
%	\caption{{Illustration of robust convex restriction and its relation to convex restriction. The x-axis shows the power dispatch at a generator, and the y-axis shows the power demand at a load. The red dot on the figure shows the nominal generation and load, and the circle/bar shows the range of uncertainty at the load bus. The original convex restriction from \cite{lee2018} (in green) is a convex subset for the feasible region of the original non-convex AC OPF (in blue) and encompasses combinations of load (x-axis) and generation (y-axis).
%	The robust convex restriction derived in this paper takes the range of load scenarios described by the uncertainty set (red circle on the y-axis) as an input and determines the range of dispatch variables (red range on the x-axis) that ensures the feasibility for all load scenarios in the uncertainty set.}}
%	\label{fig:robust_cvxrs}
%\end{figure}

\subsection{Contribution}
While the previously developed convex restrictions provide a set of admissible power injections, these restrictions are constructed without considering a priori models of the uncertain disturbances, e.g., variations of the renewables and demand forecast errors.
The robust AC OPF problem poses two unique challenges that are not considered in our previous work: First, the dispatch point needs to be robust against a range of possible load realizations in a \emph{pre-specified} uncertainty set, i.e., there needs to exist a feasible AC power flow solution for all possible realizations within the uncertainty set. This is different from a standard AC OPF problem considered in \cite{lee2018, lee2019feasible}, which treats the load as a fixed parameter. %, and from the security assessment in \cite{XX} which treats the load as a controllable variable.
Second, the dispatch point must be \emph{optimized} with respect to the generation cost, and this requires additional flexibility compared to security assessment in \cite{nguyen2018constructing}.% rather than assumed to be known as in \cite{XX}.
%The robust AC OPF problem poses two unique challenges that are not considered in our previous work: the dispatch point needs to (i)~be robust against a \emph{pre-specified} uncertainty set and (ii)~be \emph{optimized} with respect to the generation cost.
%Fig.~\ref{fig:robust_cvxrs} illustrates the requirements of the robust AC OPF problem where the operator need to choose the active power generation dispatch (on the x-axis) that minimizes the generation cost, while ensuring that the solution remains feasible for all possible variations in the load (marked as a red circle on the y-axis). 
%The resulting operating point is the red point, with the red bar to describe the impact of load variations.
%The variation at the load buses is described by an uncertainty set, which is marked by a red bar on the operating point.

To solve robust AC OPF problems, this paper introduces a \emph{robust convex restriction}, i.e., a convex sufficient condition on the dispatch variables. All points within the convex restriction are guaranteed to be \emph{robustly feasible} with respect to a given uncertainty set. In other words, all operating points within the robust convex restriction will remain solvable and feasible when perturbed by all uncertainty realizations within the specified uncertainty set.

Fig.~\ref{fig:robust_cvxrs} provides an illustrative example showing feasible spaces for an AC OPF problem, a convex restriction, and a robust convex restriction. A robust OPF dispatch solution must ensure that the solution remains feasible for all possible variations in the load, which corresponds to the red circle in Fig.~\ref{fig:robust_cvxrs}(a) and the $y$-axis in Fig.~\ref{fig:robust_cvxrs}(b).
The system operators need to choose the active power generation dispatch that minimizes the generation cost, which corresponds to the red points in Fig.~\ref{fig:robust_cvxrs}(c) and on the $x$-axis in Fig.~\ref{fig:robust_cvxrs}(b). 
%The blue region represents the non-convex feasible space, and the green region represents the convex restriction from \cite{lee2018}. 
The challenge in constructing the robust convex restriction is that we need to consider a range of possible load realizations which are random/uncontrollable, but known to lie within a pre-specified uncertainty set (red circle). The robust convex restriction takes the load uncertainty as an input and expresses the range of dispatch variables within which all load realizations are guaranteed to remain feasible. 
An example of a robust operating point is the red dot, with the red bar showing the possible load variations.
%Fig.~\ref{fig:robust_cvxrs} illustrates that any dispatch point in the robust convex restriction satisfies the condition that the given uncertainty set will always be a feasible set of the AC OPF problem (i.e., the red bar on the operating point is within the feasible region.)

The main contributions of this paper are summarized below.

\begin{enumerate}
\item We develop a \emph{robust convex restriction}, which provides a convex inner approximation to the \emph{robust} AC optimal power flow problem. 
The robust condition ensures the feasibility of all the operating constraints and the AC power flow equations for any realizations of the power injections within the specified set and hence provably guarantees robust feasibility. 
In addition, the convexity of the condition brings computational advantages.

%The robust convex restriction is a non-trivial extension from the convex restrictions proposed in~\cite{nguyen2018constructing, lee2018, lee2019feasible}, which only considers a single load profile without uncertainties from loads and renewables. 
\item We use the robust convex restriction to formulate a robust OPF problem that either (i) maximizes the uncertainty margin or (ii) minimizes the worst-case generation cost, while guaranteeing robustness against a given uncertainty set. We then propose a tractable algorithm for AC OPF problems that solves a sequence of convex optimization problems to provide solutions with provable robust feasibility guarantees. To the best of the authors' knowledge, these are the first solution methods to provide rigorous robust feasibility guarantees for nonlinear AC OPF problems with meshed topologies, without requiring controllable loads or generators at every node.

\item We demonstrate the effectiveness and tractability of this algorithm using numerical experiments on PGLib test cases~\cite{pglib}. The results illustrate the ability of our algorithm to control the trade-off between the level of robustness and the operating cost.
\end{enumerate}

This paper is organized as follows. Section~\ref{sec:model} introduces the system model. Section~\ref{sec:robust} formalizes robust feasibility. Section~\ref{sec:suff_cond} derives the sufficient condition for robust feasibility. Section~\ref{sec:algorithm} leverages this condition to develop our proposed robust AC OPF algorithm. Section~\ref{sec:experiments} empirically demonstrates this algorithm's performance. Section~\ref{sec:conclusion} concludes the paper.

\section{System Model and Preliminaries} \label{sec:model}
This section introduces the OPF problem with consideration of uncertainty in power injections. We use a distributed slack generator formulation, generalizing the model in~\cite{lee2018,lee2019feasible}. The distributed slack plays an important role in determining the generators' response to the uncertain power injections.

\subsection{Notation}
The scalars $n_\textrm{b}$, $n_\textrm{g}$, $n_\textrm{pv}$, $n_\textrm{pq}$, $n_\textrm{l}$, and $n_\textrm{d}$ denote the number of buses, generators, PV, PQ buses, lines, and loads, respectively. The variables $p_\textrm{g}\in\mathbb{R}^{n_\textrm{g}}$ and $q_\textrm{g}\in\mathbb{R}^{n_\textrm{g}}$ represent the generators' active and reactive power outputs. Uncontrollable active and reactive power injections are denoted by $p_\textrm{d}\in\mathbb{R}^{n_\textrm{d}}$ and $q_\textrm{d}\in\mathbb{R}^{n_\textrm{d}}$ where positive values indicate stochastic loads and negative values indicate uncertain generation such as renewables. 
The voltage magnitudes and phase angles are $v\in\mathbb{R}^{n_\textrm{b}}$ and $\theta\in\mathbb{R}^{n_\textrm{b}}$.
The \textit{from} and \textit{to} buses for the lines are denoted by ``$\textrm{f}$'' and ``$\textrm{t}$''.
The non-reference, PV, and PQ elements of a vector are denoted with subscripts ``$\textrm{ns}$'', ``$\textrm{pv}$'', and ``$\textrm{pq}$''. 
Let $E\in\mathbb{R}^{n_\textrm{b}\times n_\textrm{l}}$ be the incidence matrix of the grid. 
The connection matrices for generator buses and load buses are denoted by $C_\textrm{g}\in\mathbb{R}^{n_\textrm{b}\times n_\textrm{g}}$ and $C_\textrm{d}\in\mathbb{R}^{n_\textrm{b}\times n_\textrm{d}}$, respectively.
The matrices $I$ and $\mathbf{0}$ denote identity and zero matrices of appropriate size.
The vertical concatenation of vectors $a$ and $b$ is denoted by $(a,b)$.

\subsection{AC Optimal Power Flow Problem Formulation}
For notational convenience, we denote the angle differences between the terminals of the transmission lines as $\varphi$:
\begin{equation}
	\varphi_l=\theta_l^\textrm{f}-\theta_l^\textrm{t}, \hskip2em l=1,\ldots,n_\textrm{l},
\end{equation}
where $\theta_i^\textrm{f}$ and $\theta_i^\textrm{t}$ are the phase angles of the \textit{from} bus and \textit{to} bus of line $l$. The AC~OPF problem is:
\begin{equation}
\underset{x,u,\overline{s}^\textrm{f},\overline{s}^\textrm{t}}{\text{minimize}} \hskip 1em c(p_\textrm{g})=\sum_{i=1}^{n_\textrm{g}} c_i (p_{\textrm{g},i})
\label{eqn:cost}
\end{equation}\vspace*{0em}
subject to: for $k=1,\ldots,n_\textrm{b}$, \vspace*{0em}
\begin{subequations} \begin{align}\label{eqn:ACPowerflow_Pbal}
p^\textrm{inj}_k=\sum_{l=1}^{n_\textrm{l}} v^\textrm{f}_lv^\textrm{t}_l\left(G^\textrm{c}_{kl}\cos{\varphi_l}+B^\textrm{s}_{kl}\sin{\varphi_l} \right)+G_\mathit{kk}^\textrm{d}v_k^2,\\ \label{eqn:ACPowerflow_Qbal}
q^\textrm{inj}_k=\sum_{l=1}^{n_\textrm{l}} v^\textrm{f}_lv^\textrm{t}_l\left(G^\textrm{s}_{kl}\sin{\varphi_l}-B^\textrm{c}_{kl}\cos{\varphi_l} \right)-B_\mathit{kk}^\textrm{d}v_k^2,
\end{align} \label{eqn:ACPowerflow} \end{subequations}\vspace*{-1em}
\begin{subequations} \begin{align}
p_{\textrm{g},i}^\textrm{ min}\leq&p_{\textrm{g},i}\leq p_{\textrm{g},i}^\textrm{max}, \ & i&=1,\ldots,n_\textrm{g}, \label{eqn:pgenlim} \\
q_{\textrm{g},i}^\textrm{min}\leq&q_{\textrm{g},i}\leq q_{\textrm{g},i}^\textrm{max}, \ & i&=1,\ldots,n_\textrm{g}, \label{eqn:qgenlim}\\
v_i^{\textrm{min}}\leq&v_i\leq v_i^{\textrm{max}}, \ & i&=1,\ldots,n_\textrm{b} \label{eqn:vmaglim} \\
\varphi^{\textrm{min}}_l\leq&\varphi_l\leq \varphi^{\textrm{max}}_l, \ & l&=1,\ldots,n_\textrm{l}, \label{eqn:anglelim} \\
(s_{\textrm{p},l}^\textrm{f})^2+&(s_{\textrm{q},l}^\textrm{f})^2\leq (s^{\textrm{max}}_l)^2, \ & l&=1,\ldots,n_\textrm{l}, \label{eqn:linelim_f} \\ (s_{\textrm{p},l}^\textrm{t})^2+&(s_{\textrm{q},l}^\textrm{t})^2\leq (s^{\textrm{max}}_l)^2, \ & l&=1,\ldots,n_\textrm{l}. \label{eqn:linelim_t}
\end{align} \label{eqn:OPconstr}%
\end{subequations}%
where the matrices $G^\textrm{c},\ G^\textrm{s},\ B^\textrm{c},\ B^\textrm{s}\in\mathbb{R}^{n_\textrm{b}\times n_\textrm{l}}$ and $G^\textrm{d},\ B^\textrm{d}\in\mathbb{R}^{n_\textrm{b}\times n_\textrm{b}}$ are transformed admittance matrices for the respective conductance and susceptance terms. 
The exact definitions of the transformed matrices are available in \cite{lee2019feasible}. The objective $c:\mathbb{R}^{n_g}\rightarrow\mathbb{R}$ is a monotonically increasing function of the active power generation. % and is assumed to be monotonically increasing with respect to each generator's output.
The active and reactive power injections are 
$p^\textrm{inj}=C_\textrm{g}p_\textrm{g}-C_\textrm{d}p_{\textrm{d}}$ and $q^\textrm{inj}=C_\textrm{g}q_\textrm{g}-C_\textrm{d}q_{\textrm{d}}$. 
Superscripts $\textrm{max}$ and $\textrm{min}$ denote the maximum and minimum limits of the associated quantity.
Constraints~\eqref{eqn:pgenlim} and~\eqref{eqn:qgenlim} impose the generators' active and reactive power output limits. Constraints~\eqref{eqn:vmaglim} and~\eqref{eqn:anglelim} limit the voltage magnitudes and the angle differences.
Constraints~\eqref{eqn:linelim_f} and~\eqref{eqn:linelim_t} impose line flow limits where
$s_{\textrm{p},l}^\textrm{f/t}$ and $s_{\textrm{q},l}^\textrm{f/t}$ are the active and reactive power flowing into the line $l$ at the \textit{from} and \textit{to} buses, respectively, and their definitions are given in in Appendix~\ref{appendix:ineq}.

\subsection{Power Injection Uncertainty Modelling}
This paper focuses on obtaining a robustly feasible operating point with respect to uncertainty in the uncontrollable power injections, such as load variations and forecast errors from renewable energy sources. 
The variable $w=(p_d,\, q_d)\in\mathbb{R}^{2n_d}$ consists of uncertain active and reactive power injections $p_\textrm{d}\in\mathbb{R}^{n_d}$ and $q_\textrm{d}\in\mathbb{R}^{n_d}$.
The nominal value of the uncertain variable is $w^{(0)}$.
We consider a bounded uncertainty set $\mathcal{W}$ modeled with a \textit{confidence ellipsoid} containing all power injections within a ball of radius $\gamma$ centered on the nominal power injections:
\begin{equation}
\mathcal{W}(\gamma)=\{w\mid (w-w^{(0)})^T\Sigma^{-1}(w-w^{(0)})\leq \gamma^2\}.
\label{eqn:ellipsoid}
\end{equation}
The power injection covariance matrix $\Sigma\in\mathbb{R}^{2n_d\times 2n_d}$ and the radius $\gamma\in\mathbb{R}$ determine the shape and size, respectively, of the uncertainty set. 
The covariance matrix $\Sigma$ captures the correlations between power injections.
By choosing an appropriate value for $\gamma$, the confidence ellipsoid can be designed such that the probability of containing the uncertainty realization is greater than the desired threshold.
For example, if the uncertainty is drawn from a univariate normal distribution, we can ensure that 95\% of the uncertainty realizations are within the confidence ellipsoid by setting $\gamma$ to be twice the variance. 

\subsection{Active Power Generation Recourse for System Balancing}
To balance the system during variations in loads, we consider a ``distributed slack'' model where each generator adjusts its active power output to account for the system-wide power imbalance. These adjustments occur according to the generators' participation factors.
The distributed slack model is formulated via an affine control policy:
\begin{equation}
p_{\textrm{g},i}(w)=p_{\textrm{g,ref},i}+\alpha_i\Delta,
\label{eq:distslack}\end{equation}
where $p_{\textrm{g,ref},i}$ is the nominal setpoint for the generator's active power output, $\alpha_i$ is the constant participation factor for generator $i$, and $\Delta\in\mathbb{R}$ is the system-side power imbalance.
The variable $\Delta$ is implicitly defined by the AC power flow equations~\eqref{eqn:ACPowerflow_Pbal} such that the active power is balanced across the system. Fluctuations in power injections lead to changes in $\Delta$ and consequent adjustments to the active power generation according to~\eqref{eq:distslack}. The distributed slack model generalizes a single slack bus formulation, which can be retrieved by setting all participation factors to~$0$ except for the participation factor of the slack bus, which is set to~$1$.\footnote{
While our formulation allows the participation factors $\alpha$ to be decision variables, permitting this flexibility in the problem led the solvers to encounter numerical problems. Therefore, we fix the participation factors to specified quantities. Improving the numerical stability of our method with variable participation factors is a subject of our future work.
}
The generators' reactive power outputs are not directly controlled, and they are set by the demand from the grid, determined by the AC power flow equations~\eqref{eqn:ACPowerflow_Qbal}.

\subsection{Variable Definitions}
\label{subsec:variable}
The buses are divided into two types according to the standard definitions in the distributed slack model:
\begin{itemize}[leftmargin=*]
    \item PV (generator) buses: $p_\textrm{g,ref}$ and $v_\textrm{g}$ are specified by the operators; $q_\textrm{pv}$ and $\theta_\textrm{pv}$ are implicitly defined by the AC power flow equations.
    \item PQ (load) buses: $p_\textrm{d}$ and $q_\textrm{d}$ are either fixed or uncertain parameters; $v_\textrm{pq}$ and $\theta_\textrm{pq}$ are implicitly defined by the AC power flow equations.
\end{itemize}
The angle at one arbitrarily chosen reference bus is set to zero.

Based on the bus types, we categorize all variables as \textit{dispatch variables} or \textit{internal states}, which are denoted by $u\in\mathbb{R}^{2n_g}$ and $x\in\mathbb{R}^{n_b+n_{pq}}$, respectively, with the following~definitions:
\begin{itemize}[leftmargin=*]
    \item \emph{Dispatch variables} refer to controllable quantities that can be set by the system operator, specifically, the active power generation setpoint and the voltage magnitude setpoint. A \emph{dispatch point} is defined by a vector composed of dispatch variables and is denoted by $u=(p_\textrm{g,ref},\,v_\textrm{g})$. 
    \item \emph{Uncertain variables} are uncontrollable quantities that are determined by external fluctuations such as renewables and demand. Uncertain variables are modeled as variations in power injections and are denoted by $w=(p_\textrm{d},\,q_\textrm{d})$.
    \item \emph{Internal states} are physical quantities that are implicitly determined by the physics of the power grid and are computed through the AC power flow equations. These include the phase angles at the non-reference buses, the voltage magnitudes at PQ buses, and the distributed slack variable~$\Delta$. Internal states are denoted by $x=(\theta_\textrm{ns},\,v_\textrm{pq},\,\Delta)$. 
\end{itemize}

Moreover, we introduce a transformed state variable $z=(\varphi,\,v_\textrm{pq},\,\Delta)$ that converts the phase angles $\theta_\textrm{ns}$ to phase angle differences $\varphi$. The transformed state is defined as \mbox{$z=Ax$} where $A=\mathbf{blkdiag}(E^T_\textrm{ns},I_{n_\textrm{pq}\times n_\textrm{pq}},1)$.
The operation $\mathbf{blkdiag}(\cdot)$ denotes the block-diagonal matrix with the arguments forming the diagonal submatrices.
The transformed state~$z$ enables working directly with the angle differences~$\varphi$ in the power flow equations' trigonometric functions.

Finally, we use the superscript $(0)$ to indicate a nominal point of any variable when the uncertainty is equal to its nominal value. That is, the nominal state variable $x^{(0)}$ is the solution to the AC power flow equations with $w=w^{(0)}$.

\section{Robust Feasibility of the AC Optimal Power Flow: Definition and Problem Formulation}
\label{sec:robust}
In this section, we introduce uncertainty into the nominal OPF problem. We first describe the unknown-but-bounded uncertainty set and then define the robust AC OPF problem.

\subsection{Robust Feasibility Against Power Injection Uncertainty}
\label{subsec:robust_definition}
Given the uncertainty model, the robust feasibility of a dispatch point is defined by the following statement:

\begin{definition}
A dispatch point $u$ is \textit{robustly feasible} if the OPF constraints are satisfied for all realizations of $w$ within the given uncertainty set $\mathcal{W}(\gamma)$. 
That is, for all $w\in\mathcal{W}(\gamma)$, there exists $x=(\theta_\textrm{ns},\,v_\textrm{pq},\,\Delta)$ such that \eqref{eqn:ACPowerflow} and \eqref{eqn:OPconstr} are satisfied.
\label{def:robust}\end{definition}

Robust feasibility is defined for a dispatch point $u$, which corresponds to the variables in the OPF problem that are directly controllable.
Power injection fluctuations $w$ change the internal states $x$ (e.g., the voltage magnitudes and phase angles) according to the AC power flow equations. Hence, the internal states adapt to the uncertainty realizations.

With this definition of robustness in mind, we aim to determine the dispatch point that minimizes the generation cost while ensuring the existence of internal states that satisfy the operational constraints for all uncertainty realizations.

\subsection{Objective Function of the Robust AC OPF Problem}
\label{subsec:robust_prob}
In this section, we discuss the objective function and constraints of the robust AC OPF problem.
There are two candidates for the objective function:
\begin{enumerate}
\item Nominal operating cost: minimize the generation cost evaluated when the uncertainty realization is equal to its nominal value (i.e., $c^{(0)}(p_\textrm{g}) = c(p_\textrm{g})\mid_{w=w^{(0)}}$).
\item Worst-case operating cost: minimize the generation cost for the worst possible realization in the uncertainty set (i.e., $c^u(p_\textrm{g}) = \max_{w\in\mathcal{W}(\gamma)}c(p_\textrm{g})$).
\end{enumerate}

While both the nominal and worst-case generation costs are important considerations, we will choose the worst-case generation cost as the objective function in this paper. 
The objective function is a nonlinear function of the dispatch variables because the power generation, $p_g$, is a function of power imbalance $\Delta$ as shown in \eqref{eq:distslack}. The power imbalance term $\Delta$ is implicitly defined by the AC power flow equations and generation recourse and is thus a nonconvex function of the decision variables.
We will formulate the convex restriction using an upper bound on $\Delta$ as an explicit variable $\Delta^u$ to obtain a convex optimization problem.
Our numerical studies in Section \ref{sec:experiments} will evaluate both the nominal and worst-case generation costs and show that those generation costs are strongly correlated.

\subsection{Robust AC OPF Problem Formulation}

Given our choice of objective function, the robust AC~OPF problem can be cast as the following optimization problem:
\begin{equation}
\underset{x,z,u,\overline{s}^\textrm{f},\overline{s}^\textrm{t}}{\text{minimize}} \hskip 1em c^u(p_\textrm{g})
\label{eqn:cost2}
\end{equation}
subject to: for all uncertainty realizations $w=(p_\textrm{d},q_\textrm{d})$ in $\mathcal{W}(\gamma)$, there exists internal states $x=(\theta_\textrm{ns},\,v_\textrm{pq},\,\Delta)$ such that the following three conditions hold: 

\noindent1)~active power balance for $k=1,\ldots,n_\textrm{b}$,
\begin{equation} \begin{aligned}
&\sum_{i=1}^{n_\textrm{g}} C_{\textrm{g},ki}(p_{\textrm{g},\textrm{ref},i}+\alpha_i\Delta)-\sum_{i=1}^{n_\textrm{d}} C_{\textrm{d},ki}\, p_{\textrm{d},i} \\
&\hskip3em=\sum_{l=1}^{n_\textrm{l}} v^\textrm{f}_lv^\textrm{t}_l\left(G^\textrm{c}_{kl}\cos{\varphi_l}+B^\textrm{s}_{kl}\sin{\varphi_l} \right)+G_\mathit{kk}^\textrm{d}v_k^2,\\ \end{aligned} \label{eqn:ACPowerflow_fp} \end{equation}
2)~reactive power balance for every PQ bus $k$,
\begin{equation}
-\sum_{i=1}^{n_\textrm{d}} C_{\textrm{d},ki}\, q_{\textrm{d},i}=\sum_{l=1}^{n_\textrm{l}} v^\textrm{f}_lv^\textrm{t}_l\left(G^\textrm{s}_{kl}\sin{\varphi_l}-B^\textrm{c}_{kl}\cos{\varphi_l}\right)-B_\mathit{kk}^\textrm{d}v_k^2,
\label{eqn:ACPowerflow_fq} \end{equation}
3)~the internal states and control variables satisfy the operational constraints in \eqref{eqn:OPconstr}. 

Equations~\eqref{eqn:ACPowerflow_fp} and~\eqref{eqn:ACPowerflow_fq} are substitutions of the definitions for the power injections and generation recourse into the AC power flow equations. 
Note that Equations~\eqref{eqn:ACPowerflow_fp} and~\eqref{eqn:ACPowerflow_fq} are sufficient to determine the internal states by solving a system of equations and unknowns of size $n_\textrm{b}+n_\textrm{pq}$.
The AC power flow equations for reactive power generations at PV buses are directly substituted into the reactive power capacity limits since they are not necessary for capturing the relationship between dispatch variables and internal states.

We require that the robust solution satisfies the operational constraints for all internal states that are realizable by the uncertainty set.

\subsection{Basis Function Formulation of the Power Flow Equations}
\label{subsec:basisF}
In this section, we rewrite the AC power flow equations in terms of basis functions that serve as building blocks for the power flow nonlinearities. The vector of nonlinear functions, $\psi:(\mathbb{R}^{n_\textrm{b}+n_\textrm{pq}},\mathbb{R}^{2n_\textrm{g}})\rightarrow\mathbb{R}^{n_\textrm{g}+2n_\textrm{l}+n_\textrm{b}}$, denotes the basis function,
$\psi(x,u)=(\psi^\mathrm{p},\, {\psi^\mathrm{cos}},\, {\psi^\mathrm{sin}},\, {\psi^\mathrm{quad}})$, where
\begin{subequations} \begin{align}
\psi_i^\mathrm{p}(x,u)&=p_{\textrm{g,ref},i}+\alpha_i\Delta, \ & i&=1,\ldots,n_\textrm{g}, \\
\psi_l^\mathrm{cos}(x,u)&=v^\mathrm{f}_lv^\mathrm{t}_l\cos{(\varphi_l)}, \ & l&=1,\ldots,n_\textrm{l}, \\
\psi_l^\mathrm{sin}(x,u)&=v^\mathrm{f}_lv^\mathrm{t}_l\sin{(\varphi_l)}, \ & l&=1,\ldots,n_\textrm{l}, \\
\psi_k^\mathrm{quad}(x,u)&=v_k^2, \ & k&=1,\ldots,n_\textrm{b}.
\end{align}\label{eqn:basisfunc}\end{subequations}

The AC power flow equations in \eqref{eqn:ACPowerflow} can be written in terms of the basis functions and the uncertain variables as 
\begin{equation}
M\psi(x,u)+Rw=0,
\label{eqn:opf_eqn}
\end{equation}
where $M\in\mathbb{R}^{(n_\textrm{b}+n_\textrm{pq})\times(n_\textrm{g}+2n_\textrm{l}+n_\textrm{b})}$ and $R\in\mathbb{R}^{(n_\textrm{b}+n_\textrm{pq})\times2n_\textrm{d}}$ are constant matrices defined as
\begin{equation}
M = \begin{bmatrix}
C_\textrm{g} & \!\!\!-G^\textrm{c} & \!\!\!-B^\textrm{s} & \!\!\!-G^\mathrm{d} \\
\mathbf{0} & B^\textrm{c}_\textrm{pq} & \!\!\!-G^\textrm{s}_\textrm{pq} & B^\mathrm{d}_\textrm{pq}
\end{bmatrix},\hskip0.5em
R = -\begin{bmatrix} C_\textrm{d} & \mathbf{0} \\ \mathbf{0} & C_\textrm{d,pq} \end{bmatrix}.
\label{eqn:pf_eqn}
\end{equation}

The matrix $B^\textrm{c}_\textrm{pq}\in\mathbb{R}^{n_\textrm{pq}\times n_\textrm{l}}$ is a submatrix of $B^\textrm{c}\in\mathbb{R}^{n_\textrm{b}\times n_\textrm{l}}$ containing the rows corresponding to the PQ buses. Matrices $G^\textrm{s}_\textrm{pq}$, $B^\textrm{d}_\textrm{pq}$, $C_\textrm{g,pq}$, and $C_\textrm{d,pq}$ are defined similarly. 

\section{Robust Convex Restriction}
\label{sec:suff_cond}
Robust feasibility of a solution requires that any power injections within the uncertainty set are feasible with respect to the non-convex feasible set of the OPF problem.
In this section, we present the robust convex restriction, which provides a convex sufficient condition for certifying that the dispatch point is robustly feasible for the AC OPF problem.
The condition generalizes convex restriction presented in in~\cite{lee2018,lee2019feasible}, which is a special case where the problem is deterministic (i.e. the uncertainty set contains only one point such that $\mathcal{W}(0)=\{w^{(0)}\}$).

\subsection{Preliminaries for Convex Restriction}
This section extends the concept of convex restriction in~\cite{lee2018,lee2019feasible} in order to incorporate the consideration of uncertain variables. 
A convex restriction is a convex inner approximation of the feasible set in the space of dispatch variables. 
The derivation of this inner approximation uses a fixed-point representation of the AC power flow equations and envelopes for the nonlinear functions.

\subsubsection{Fixed-Point Representation of the AC Power Flow Equations}
Let us define the function $f(x,u,w)=M\psi(x,u)+Rw$.
A solution to the power flow equations~\eqref{eqn:opf_eqn} can be interpreted as a zero for a system of nonlinear equations, i.e., $f(x,u,w)=0$. This set of nonlinear equations can be converted to a fixed-point form, which corresponds to Newton's iteration,
\begin{equation}\begin{aligned}
    x&=-J_{f,(0)}^{-1}Mg(x,u)-J_{f,(0)}^{-1}Rw.
\end{aligned}\label{eq:fixed-point-form}\end{equation}
Here $g:(\mathbb{R}^{n_\textrm{b}+n_\textrm{pq}},\mathbb{R}^{2n_\textrm{g}})\rightarrow\mathbb{R}^{3n_\textrm{b}+2n_\textrm{l}}$ is the residual of the nonlinear basis functions,
\begin{equation}
    g(x,u) = \psi(x,u) - J_{\psi,(0)}x.
    \label{eqn:residaul}
\end{equation}
The matrices $J_{f,(0)}\in\mathbb{R}^{(n_\textrm{b}+n_\textrm{pq})\times(n_\textrm{b}+n_\textrm{pq})}$ and  $J_{\psi,(0)}\in\mathbb{R}^{(n_\textrm{g}+n_\textrm{pq}+n_\textrm{b})\times(n_\textrm{b}+n_\textrm{pq})}$ are Jacobians evaluated at the nominal operating point, defined by
\begin{equation}\begin{aligned}
    J_{f,(0)}=\nabla_x f\mid_{(x^{(0)},u^{(0)})},\hskip1em %\\
    J_{\psi,(0)}=\nabla_x \psi\mid_{(x^{(0)},u^{(0)})}.
\end{aligned}\label{eqn:J_eval}\end{equation} 
These Jacobians are related by $J_{f,(0)}=MJ_{\psi,(0)}$. Equations~\eqref{eqn:opf_eqn} and~\eqref{eq:fixed-point-form} are equivalent given that the power flow Jacobian $J_{f,(0)}$ is non-singular.

\subsubsection{Upper Convex and Lower Concave Envelopes}
Denote the $k$-th element of the residual $g(x,u)$ as $g_k(x,u
)$. Upper convex and lower concave envelopes of the function $g_k$ are defined using lower and upper bounding functions $g^\ell_k(x,u)$ and $g^u_k(x,u)$, respectively.
Each envelope is defined such that the following condition holds for all $x$ and $u$ within the operational limits:

\begin{equation}
    g^\ell_k(x,u)\leq g_k(x,u)\leq g^u_k(x,u),
    \label{eq:envelope}
\end{equation}
where $g^u_k(x,u)$ is a convex function and $g^\ell_k(x,u)$ is a concave function with respect to $x$ and $u$. Similarly, let~$\psi^\ell$ and~$\psi^u$ define upper convex and lower concave envelopes for~$\psi$. An example of upper convex and lower concave envelopes is shown by the solid red lines in Fig.~\ref{fig:envelope}. Expressions for the concave envelopes used in AC OPF problems are given in~\cite{lee2018}.

\begin{figure}[!htbp]
	\centering
	\includegraphics[width=3in]{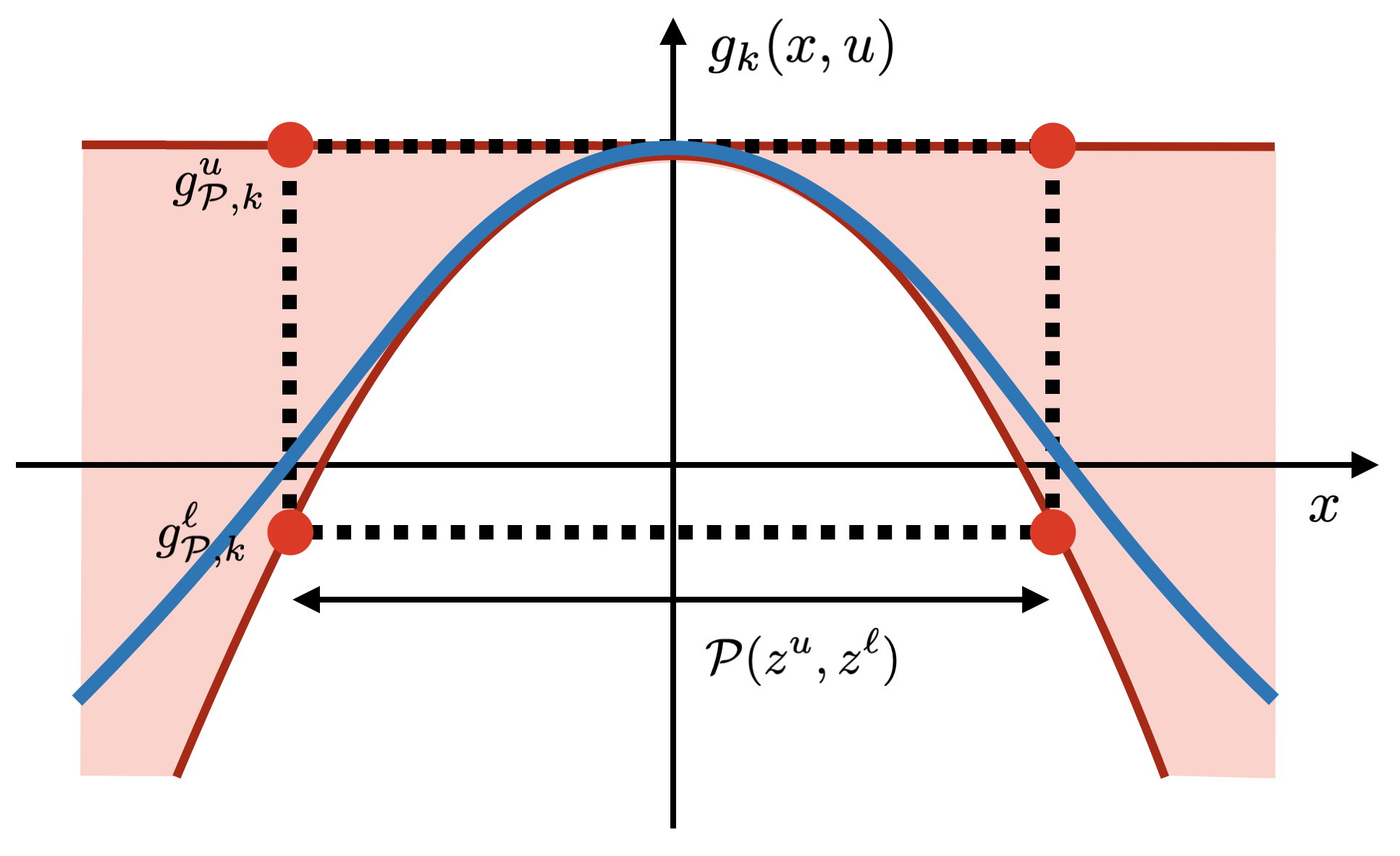}
	\caption{Illustration of upper convex and lower concave envelopes of the cosine function. The function $g_k(x,u)$ in blue is bounded by the upper convex and lower concave envelopes in red. $\mathcal{P}(z^u,z^\ell)$ is the set of possible internal states and the red dots represent the upper and lower bounds on $g_k(x,u)$ over $\mathcal{P}(z^u,z^\ell)$.}
	\label{fig:envelope}
\end{figure}

\subsubsection{Outer Approximations of the Internal States}
The internal states adapt to both the uncertainty realizations and the choice of dispatch variables.
To bound the values taken by the internal states, 
we define an outer-approximation of the set of internal states as 
\begin{equation}
\mathcal{P}(z^u,z^\ell)=\left\{x\bigm| z^\ell\leq Ax\leq z^u\right\},
\label{eqn:P}
\end{equation}
which is described by the parameter $z^u,z^\ell\in\mathbb{R}^{n_\textrm{l}+n_\textrm{pq}+1}$, and the matrix $A$ was defined in Section \ref{subsec:variable}. 
This set can be equivalently written as
$\mathcal{P}(z^u,\,z^\ell)=\{x\mid \varphi^\ell\leq E_\textrm{ns}\theta_\textrm{ns}\leq\varphi^u,\; v^\ell_\textrm{pq}\leq v_\textrm{pq}\leq v^u_\textrm{pq},\;\Delta^\ell\leq\Delta\leq\Delta^u\}$. 
The convex restriction will be used to verify that the set $\mathcal{P}(z^u,\, z^\ell)$ contains internal state solutions to the AC power flow equations for all uncertainty realizations.
Later, we will treat the upper and lower bounds $z^u$ and $z^\ell$ as decision variables.

\subsubsection{Bounds on Residuals over $\mathcal{P}(z^u,z^\ell)$}
We next define variables that represent over- and under-estimators of the residual, $g_k(x,u)$, over the set $\mathcal{P}(z^u,z^\ell)$. These over- and under-estimators are denoted by $g^u_{\mathcal{P},k}$ and $g^\ell_{\mathcal{P},k}$, respectively, and are defined such that for all $x\in\mathcal{P}(z^u,z^\ell)$ and $u$ within the operational limits,
\begin{equation}
g^\ell_{\mathcal{P},k}\leq g_{k}(x,u)\leq g^u_{\mathcal{P},k}, \hskip1em \forall\, x\in\mathcal{P}(z^u,z^\ell),\, \forall\,u.
\label{eqn:bounds}
\end{equation}

The inequality \eqref{eqn:bounds} needs to be satisfied for a set of continuous variables $x$ and thus corresponds to an infinite number of constraints.
For this inequality to become tractable, we need to replace \eqref{eqn:bounds} with a finite set of deterministic constraints. 
The following lemma uses the upper convex and lower concave envelopes from \eqref{eq:envelope} to enforce
\begin{equation}
    g^\ell_{\mathcal{P},k}\leq g^\ell_k(x,u), \hskip1em g^u_k(x,u)\leq g^u_{\mathcal{P},k}
\end{equation}
for all $x\in\mathcal{P}(z^u,z^\ell)$ and $u$ in order to provide a set of deterministic convex constraints for \eqref{eqn:bounds}.

\begin{lemma}
Suppose that for all vertices of the set $\mathcal{P}(z^u,\,z^\ell)$ (i.e., $\forall$ $v_l^\textrm{f}\in\{v_l^{\textrm{f},\ell},v_l^{\textrm{f},u}\}$, $v_l^\textrm{t}\in\{v_l^{\textrm{t},\ell},v_l^{\textrm{t},u}\}$, $\varphi_l\in\{\varphi_l^\ell,\varphi_l^u\}$, and $v_k\in\{v_k^\ell,v_k^u\}$), the following inequalities are satisfied:
\begin{subequations} \begin{align}
g_{\mathcal{P},i}^{\mathrm{p},\ell}&\leq g_i^{\mathrm{p},\ell}(p_{\textrm{g,ref},i}),
&g_i^{\mathrm{p},u}(p_{\textrm{g,ref},i})&\leq g_{\mathcal{P},i}^{\mathrm{p},u},  \label{eqn:g_P_a} \\
g_{\mathcal{P},l}^{\mathrm{cos},\ell}&\leq g_l^{\mathrm{cos},\ell}(v_l^\textrm{f},v_l^\textrm{t},\varphi_l), 
&g_l^{\mathrm{cos},u}(v_l^\textrm{f},v_l^\textrm{t},\varphi_l)&\leq g_{\mathcal{P},l}^{\mathrm{cos},u}, \label{eqn:g_P_b} \\
g_{\mathcal{P},l}^{\mathrm{sin},\ell}&\leq g_l^{\mathrm{sin},\ell}(v_l^\textrm{f},v_l^\textrm{t},\varphi_l),
&g_l^{\mathrm{sin},u}(v_l^\textrm{f},v_l^\textrm{t},\varphi_l)&\leq g_{\mathcal{P},l}^{\mathrm{sin},u} , \label{eqn:g_P_c} \\
g_{\mathcal{P},k}^{\mathrm{v},\ell}&\leq g_k^{\mathrm{v},\ell}(v_k),
&g_k^{\mathrm{v},u}(v_k)&\leq g_{\mathcal{P},k}^{\mathrm{v},u}, \label{eqn:g_P_d}
\end{align} \label{eqn:g_P}\end{subequations}%
for $i=1,\ldots,n_\textrm{g}$, $l=1,\ldots,n_\textrm{l}$, and $k=1,\ldots,n_\textrm{b}$.
Then the nonlinear residual bounds $g_{\mathcal{P}}^u$ and $g_{\mathcal{P}}^\ell$ satisfy \begin{equation}
g_{\mathcal{P},k}^\ell\leq g_{k}(x,u)\leq g_{\mathcal{P},k}^u, \hskip1em \forall\, x\in\mathcal{P}(z^u,z^\ell),\, \forall\,u.
\end{equation}
\label{lemma:g_bound}
\end{lemma}

The proof for Lemma 1 is in the Appendix.
The residual terms $g^\mathrm{p}$, $g^\mathrm{cos}$, $g^\mathrm{sin}$, $g^\mathrm{v}$ are defined according to \eqref{eqn:residaul} where the non-linear functions $\psi$ are defined in \eqref{eqn:basisfunc}.
Note that \eqref{eqn:g_P_a} is only a function of $p_\textrm{g,ref}$ since the linear term associated with $\Delta$ gets subtracted in \eqref{eqn:residaul}. The number of constraints for \eqref{eqn:g_P_a}, \eqref{eqn:g_P_b}, \eqref{eqn:g_P_c}, and \eqref{eqn:g_P_d} are $n_\textrm{g}$, $2^3\cdot n_\textrm{l}$, $2^3\cdot n_\textrm{l}$, and $2\cdot n_\textrm{b}$, respectively.

Lemma \ref{lemma:g_bound} states that the upper and lower bounds of the function $g_k(x,u)$ over the set $\mathcal{P}(z^u,\,z^\ell)$ occur at one of the extreme points of the set $\mathcal{P}(z^u,\,z^\ell)$.
Fig.~\ref{fig:envelope} provides an illustration where the black dashed rectangle contains the possible realizations of the function from the uncertainty set. The x-axis of the dashed rectangle represents the set $\mathcal{P}(z^u,\,z^\ell)$, which is the range of the internal states that can be realized from the uncertainty set.
The y-axis of the dashed rectangle represents the range of the nonlinear function $g_k(x,u)$ that can be realized from the uncertainty set. The bounds on the nonlinear function $g_k(x,u)$ over the uncertainty set realization can be obtained from the upper convex and lower concave envelope evaluated at the vertices of $\mathcal{P}(z^u,\,z^\ell)$.
This is shown as the vertices marked by the red dots in Fig.~\ref{fig:envelope}.
The following section derives a condition which ensures that $\mathcal{P}(z^u,z^\ell)$ is an outer approximation of the internal states.

\subsection{Robustness Condition from Fixed-Point Representation}
In this section, we present our approach for certifying robustness via a fixed-point argument. 
Robust feasibility requires that every realization in a \emph{pre-specified} uncertainty set has a corresponding solution to the AC power flow equations which satisfies the operational limits.
The problem in its original form cannot be expressed as a deterministic optimization problem since the equality and inequality constraints need to be satisfied for a set of uncertain variables.

There is a convenient way to handle robustness constraints via a fixed-point representation. We use the following lemma based on Brouwer's fixed-point theorem.

\begin{lemma}
Suppose that $\mathcal{P}\subseteq\mathbb{R}^{n_\textrm{b}+n_\textrm{pq}}$ is a non-empty, compact, convex set, and $G_{u,w}:\mathcal{P}\rightarrow\mathcal{P}$ is a continuous map for all $w\in\mathcal{W}$. Then for any $w\in\mathcal{W}$, there exists some $x\in\mathcal{P}$ such that $G_{u,w}(x)=x$.
\label{thm:Brouwer}
\end{lemma}

Lemma \ref{thm:Brouwer} extends Brouwer's fixed-point theorem \cite{brouwer1911abbildung} to the nonlinear map $G_{u,w}$, parameterized by both the dispatch variables and the uncertain variables.
If the set $\mathcal{P}$ is self-mapping (i.e., $\forall x\in\mathcal{P}$, $G_{u,w}(x)\in\mathcal{P}$) with respect to the fixed-point equation, then $\mathcal{P}$ provides an outer-approximation of the points $x$ that are realizable by the uncertainty set $\mathcal{W}$.
Using the fixed-point equation defined in \eqref{eq:fixed-point-form}, the condition for self-mapping can be expressed as
\begin{equation}\begin{aligned}
    \forall x\in\mathcal{P},\hskip0.2em \forall w\in\mathcal{W},\hskip0.5em -J_{f,(0)}^{-1}Mg(x,u)-J_{f,(0)}^{-1}Rw\in\mathcal{P}.
\label{eqn:self_mapping}
\end{aligned}\end{equation}
Substituting the set $\mathcal{P}(z^u,\,z^\ell)$ defined in \eqref{eqn:P} into \eqref{eqn:self_mapping} yields
\begin{equation}\begin{aligned}
    \max_{x\in\mathcal{P},\; w\in\mathcal{W}} A_i\left(-J_{f,(0)}^{-1}Mg(x,u)-J_{f,(0)}^{-1}Rw\right) &\leq z^u_i, \\
    \min_{x\in\mathcal{P},\; w\in\mathcal{W}} A_i\left(-J_{f,(0)}^{-1}Mg(x,u)-J_{f,(0)}^{-1}Rw\right) &\geq z^\ell_i.
    \label{eqn:self_mapping_P}
\end{aligned}\end{equation}
for $i=1,\ldots,n_\textrm{l}+n_\textrm{pq}+1$.
The self-mapping condition in \eqref{eqn:self_mapping_P} is expressed by $2\cdot(n_\textrm{l}+n_\textrm{pq}+1)$ inequality constraints.
The number of inequality constraints for the self-mapping condition scales linearly with the system size because we design the outer approximation $\mathcal{P}(z^u,z^\ell)$ in \eqref{eqn:P} as an intersection of intervals.

The condition in \eqref{eqn:self_mapping_P} requires that an inequality is satisfied over the sets $\mathcal{P}$ and $\mathcal{W}$.
We next convert \eqref{eqn:self_mapping_P} to a deterministic constraint by bounding the left-hand side using envelopes and known analytical solutions for optimization problems.

\subsection{Sufficient Condition for Robust Feasibility}
In this section, we use the machinery developed so far to derive a convex sufficient condition for robust feasibility of a dispatch point $u$ against the uncertainty set $\mathcal{W}(\gamma)$. 
The condition allows us to solve a deterministic convex optimization problem to obtain a robustly feasible solution for the robust optimization problem in Section~\ref{subsec:robust_prob}. The following lemmas show analytical solutions to optimization problems that will be used to convert robustness constraints to deterministic constraints.

\begin{lemma}
The maximum of a linear function over a bounded set of intervals has the following analytical solution:
\begin{equation}
\max_{x\in[x^\ell, x^u]} c^Tx = (c^+)^T x^u + (c^-)^T x^\ell.
\end{equation}
where $c\in\mathbb{R}^n$ is a constant cost vector, and $c^+, c^- \in\mathbb{R}^n$ are defined as $c_i^+ = \max\{c_i,0\}$ and $c_i^- = \min\{c_i,0\}$ for all $i$.
\label{lemma:max_interval}
\end{lemma}

\begin{lemma}
The maximum of a linear function over a bounded ellipsoid has the following analytical solution:
\begin{equation}
\max_{w\in\mathcal{W}(\gamma)} c^Tw
=c^T\,w^{(0)}+\gamma\lVert c^T\,\Sigma^{1/2}\rVert_2,
\label{eqn:robust_analytical}\end{equation}
where $\mathcal{W}(\gamma)$ is the confidence ellipsoid defined in \eqref{eqn:ellipsoid}, and $\Sigma^{1/2}$ is the Cholesky decomposition of $\Sigma$ such that $\Sigma = (\Sigma^{1/2})(\Sigma^{1/2})^T$.
\label{lemma:max_ellipsoid}
\end{lemma}

The proofs of Lemma \ref{lemma:max_interval} and \ref{lemma:max_ellipsoid} are provided in Appendices~\ref{appendix:max_interval} and \ref{appendix:max_ellipsoid}.
These lemmas are used to upper bound the left-hand side of inequality \eqref{eqn:self_mapping_P} where Lemma \ref{lemma:max_interval} is used to bound the term $-A_iJ_{f,(0)}^{-1}Mg(x,u)$, and Lemma 3 is used to bound the term $-A_iJ_{f,(0)}^{-1}Rw$.
To simplify our notation, let us define the constant matrix $K$ by 
\begin{equation}
    K=-AJ_{f,0}^{-1}M,
\end{equation}
and let the matrices $K^+,\,K^-\in\mathbb{R}^{(n_\textrm{l}+n_\textrm{pq}+1)\times (n_\textrm{g}+2n_\textrm{l}+n_\textrm{b})}$ be $K^+_{ij}=\max\{K_{ij},0\}$ and $K^-_{ij}=\min\{K_{ij},0\}$ for each element of $K$. The following theorem provides a convex sufficient condition that ensures $\mathcal{P}$ is an outer-approximation of the possible realizations of the internal states.

\begin{theorem}{(Robust Solvability Condition for AC Optimal Power Flow)}
Given a dispatch point $u=(p_{g,\textrm{ref}},\, v_\textrm{g})$, suppose that there exist internal state bounds $z^u$, $z^\ell$ and residual bounds $g^u_\mathcal{P}$, $g^\ell_\mathcal{P}$ that satisfy the inequality conditions in \eqref{eqn:g_P} and
\begin{equation}\begin{aligned}
K^+ g^u_\mathcal{P}+K^- g^\ell_\mathcal{P}+\xi^u(\gamma) &\leq z^u, \\
K^+ g^\ell_\mathcal{P}+K^- g^u_\mathcal{P}+\xi^\ell(\gamma) &\geq z^\ell,
\end{aligned}\label{eqn:conv_restr_eq}\end{equation}
where the margins $\xi_i(\gamma)$ and $\zeta_i(\gamma)$ are linear functions of $\gamma$:
\begin{equation}\begin{aligned}
\xi^u_i(\gamma)=-A_iJ_{f,0}^{-1}Rw^{(0)}+\gamma\left\lVert A_iJ_{f,0}^{-1}R\,\Sigma^{1/2}\right\rVert_2, \\
\xi^\ell_i(\gamma)=-A_iJ_{f,0}^{-1}Rw^{(0)}-\gamma\left\lVert A_iJ_{f,0}^{-1}R\,\Sigma^{1/2}\right\rVert_2. \\
\end{aligned}\end{equation}
Then for every $w\in\mathcal{W}(\gamma)$, there exists an internal state solution that satisfies the AC power flow equations and the solution can be outer approximated by $\mathcal{P}(z^u,\,z^\ell)$ (i.e., $x\in\mathcal{P}(z^u,\,z^\ell)$).
\label{thm_robustness}
\end{theorem}

The proof of Theorem \ref{thm_robustness} is provided in Appendix~\ref{appendix:proof}. Equation~\eqref{eqn:conv_restr_eq} ensures that the AC power flow equations have a solution with a corresponding internal state $x$ within $\mathcal{P}(z^u,\,z^\ell)$ for every uncertainty realization. Next, we derive a convex condition ensuring that all internal states in the outer approximation $\mathcal{P}(z^u,\,z^\ell)$ satisfy the operational constraints.

%and line flows $(\overline{s}^\textrm{f},\overline{s}^\textrm{t})$

\begin{theorem}{(Robust Feasibility Condition for AC Optimal Power Flow)}
The dispatch point $u=(p_{g,\textrm{ref}},\, v_\textrm{g})$ is robustly feasible with respect to the uncertainty set $\mathcal{W}(\gamma)$ if there exist bounds on the internal states $z^u=(\varphi^u,\,v^u_\textrm{pq},\,\Delta^u)$, $z^\ell = (\varphi^\ell,\,v^\ell_\textrm{pq},\Delta^\ell)$, $g^u_\mathcal{P}$, $g^\ell_\mathcal{P}$, $\psi^u_\mathcal{P}$, and $\psi^\ell_\mathcal{P}$ that satisfy \eqref{eqn:g_P}, \eqref{eqn:conv_restr_eq}, and the following set of operational constraints:
\begin{subequations} \begin{align}
p_{\textrm{g},i}^\textrm{ min}\leq&p_{\textrm{g},\textrm{ref},i}+\alpha\Delta^\ell, \ & i&=1,\ldots,n_\textrm{g}, \\
&p_{\textrm{g},\textrm{ref},i}+\alpha\Delta^u\leq p_{\textrm{g},i}^\textrm{max}, \ & i&=1,\ldots,n_\textrm{g}, \\
q_{\textrm{g},i}^\textrm{min}\leq&q_{\textrm{g},i}^\ell,\hskip1.12em q_{\textrm{g},i}^u\leq q_{\textrm{g},i}^\textrm{max}, \ & i&=1,\ldots,n_\textrm{g},\\
v_{\textrm{pq},i}^{\textrm{min}}\leq&v_{\textrm{pq},i}^\ell,\hskip0.4em v_{\textrm{pq},i}^u\leq v_{\textrm{pq},i}^{\textrm{max}}, \ & i&=1,\ldots,n_\textrm{pq} \\
v_{\textrm{g},i}^{\textrm{min}}\leq&v_{\textrm{g},i}\leq v_{\textrm{g},i}^{\textrm{max}}, \ & i&=1,\ldots,n_\textrm{g} \\
\varphi^{\textrm{min}}_l\leq&\varphi_l^\ell,\hskip1.7em \varphi_l^u\leq \varphi^{\textrm{max}}_l, \ & l&=1,\ldots,n_\textrm{l}, \\
(s_{\textrm{p},l}^{\textrm{f},u})^2+&(s_{\textrm{q},l}^{\textrm{f},u})^2\leq (s^{\textrm{max}}_l)^2, \ & l&=1,\ldots,n_\textrm{l}, \\ (s_{\textrm{p},l}^{\textrm{t},u})^2+&(s_{\textrm{q},l}^{\textrm{t},u})^2\leq (s^{\textrm{max}}_l)^2, \ & l&=1,\ldots,n_\textrm{l}.
\end{align}
\label{thm:operational_constr}\end{subequations}
The reactive power generation bounds $q_\textrm{g}^\ell$ and $q_\textrm{g}^u$ satisfy
\begin{equation}\begin{aligned}
L_\textrm{q}^+\psi^u_\mathcal{P}+L_\textrm{q}^-\psi^\ell_\mathcal{P}+\zeta^u(\gamma) &\leq C_\textrm{g,pv}\, q_g^u,\\
L_\textrm{q}^-\psi^u_\mathcal{P}+L_\textrm{q}^+\psi^\ell_\mathcal{P}+\zeta^\ell(\gamma) &\geq C_\textrm{g,pv}\, q_g^\ell,
\end{aligned} \label{eqn:conv_restr_ineq1} \end{equation}
\begin{equation}\begin{aligned}
    \zeta_i^u(\gamma)&=C_{\textrm{d,pv},i}\Sigma^{1/2}_\textrm{q}w^{(0)}+\gamma\left\lVert C_{\textrm{d,pv},i}\Sigma^{1/2}_\textrm{q}\right\rVert_2, \\
    \zeta_i^\ell(\gamma)&=C_{\textrm{d,pv},i}\Sigma^{1/2}_\textrm{q}w^{(0)}-\gamma\left\lVert C_{\textrm{d,pv},i}\Sigma^{1/2}_\textrm{q}\right\rVert_2,
    %\zeta_i(\gamma)&=\gamma\left\lVert \begin{bmatrix} \mathbf{0}_{n_b\times n_d} & C_{\textrm{d},i} \end{bmatrix}\Sigma^{1/2}\right\rVert_2.
\end{aligned} \label{eqn:zeta}\end{equation}
and the line flow bounds $s_\textrm{p}^{\textrm{t},u}$, $s_\textrm{q}^{\textrm{f},u}$, and $s_\textrm{q}^{\textrm{t},u}$ satisfy
\begin{equation}\begin{aligned}
L_\textrm{j, line}^{\textrm{k},+}\psi^u_\mathcal{P}+L_\textrm{j,line}^{\textrm{k},-}\psi^\ell_\mathcal{P}&\leq s_\textrm{j}^{\textrm{k},u},\\
-L_\textrm{j,line}^{\textrm{k},-}\psi^u_\mathcal{P}-L_\textrm{j,line}^{\textrm{k},+}\psi^\ell_\mathcal{P}&\leq s_\textrm{j}^{\textrm{k},u}.
\end{aligned}\label{eqn:conv_restr_ineq2}\end{equation}
for $\textrm{k}\in\{\textrm{f},\textrm{t}\}$ and $\textrm{j}\in\{\textrm{p},\textrm{q}\}$. The decision variables $\psi^u_\mathcal{P}$ and $\psi^\ell_\mathcal{P}$ are the basis function bounds over $\mathcal{P}(z^u,\,z^\ell)$, and they are constrained by  \eqref{eqn:g_P} in Lemma \ref{lemma:g_bound} by replacing the function $g$ by $\psi$.
\label{thm:thm2}\end{theorem}

The proof of Theorem \ref{thm:thm2} is provided in Appendix~\ref{appendix:proof2}.
The matrices $L^+$ and $L^-$ are defined similarly to $K^+$ and $K^-$ where $L_\textrm{q}$ and $L_\textrm{line}$ are defined in Appendix \ref{appendix:ineq}.
The matrices $C_\textrm{d,pv}\in\mathbb{R}^{n_\textrm{pv}\times n_\textrm{d}}$ and $C_\textrm{g,pv}\in\mathbb{R}^{n_\textrm{pv}\times n_\textrm{g}}$ are submatrices of connection matrices that select the rows corresponding to PV buses.
Theorems~\ref{thm_robustness} and \ref{thm:thm2} generalize the convex restriction in~\cite{lee2018,lee2019feasible}, which is retrieved by setting $\gamma = 0$ (i.e., the case where the uncertainty set only contains the nominal power injections).
Equation~\eqref{eqn:conv_restr_eq} shows that robust feasibility is guaranteed by introducing extra margins, $\xi(\gamma)$ and $\zeta(\gamma)$, into the convex restriction condition.

The robust feasibility condition is convex with respect to all decision variables (e.g., $u$, $z^u$, $z^\ell$) and the uncertainty set size parameter $\gamma$. 
Our condition is formulated as a system of convex quadratic constraints for which there exists well-developed theory and algorithms for solving the resulting optimization problem. There are commercial solvers such as Mosek, Gurobi, and CPLEX that can be used to solve convex quadratically constrained quadratic programming problems. Moreover, the number of associated constraints scales linearly with the number of buses and the number of lines in the system, making the approach tractable for large systems.

\section{Algorithms for Robust OPF Problems}
\label{sec:algorithm}

OPF formulations seek the dispatch point with minimum generation cost while considering the load demands and renewable generation.
We consider two important questions:
\begin{itemize}
    \item What dispatch point is robust with respect to the largest uncertainty set?
    \item How can we compute a low-cost dispatch point that is robust with respect to a given uncertainty set?
\end{itemize}
We next develop optimization formulations and solution algorithms that use the previously presented robust feasibility condition in order to rigorously answer both of these questions.

\subsection{Maximizing the Robustness Margin}
We first consider a setting where the system operator seeks a dispatch point $u$ that maximizes the robustness margin.
The robustness margin is defined as the size of the uncertainty set against which the solution is robust, which in our case can be parametrized as the radius of the uncertainty set $\gamma$. When the system operator wants to maximize security with respect to uncertain power injections, the problem can be formulated as
\begin{equation}
    \begin{aligned}
        \underset{u,\,\gamma,\,z^u,\,z^\ell,\,g_\mathcal{P}^u,\,g_\mathcal{P}^\ell}{\text{maximize}} \hskip 1em & \gamma \\
        \text{subject to} \hskip 1em & \text{\eqref{eqn:g_P},\;\eqref{eqn:conv_restr_eq}--\eqref{eqn:conv_restr_ineq2}}.
        %\; \text{and}\; u\in\mathcal{H}
        %u=u^{(0)}
    \end{aligned}
	\label{eqn:OPT_robust_feasibility}
\end{equation}
Since the constraints in~\eqref{eqn:OPT_robust_feasibility} include the sufficient condition for robustness, the solution $u$ is robust against any uncertainty realization $w\in\mathcal{W}(\gamma)$. The margin $\gamma$ computed by solving a convex optimization problem in~\eqref{eqn:OPT_robust_feasibility} is a guaranteed lower bound on the robustness margin.
In the case where the dispatch point is already determined, the problem in~\eqref{eqn:OPT_robust_feasibility} can be solved with an additional constraint $u=u^{(0)}$ where $u^{(0)}$ is the determined dispatch point.

\subsection{Robust AC OPF Algorithm}
\label{subsec:robust_acopf_algorithm}
We next consider a setting where the system operator wants to solve the robust AC~OPF problem with a robustness margin of at least $\gamma_\textrm{req}$.
Robust optimization is a ``worst-case'' approach where the constraints need to be satisfied for all uncertainty realizations and the objective function minimizes the maximum cost among all uncertainty realizations:

\begin{equation}
    \begin{aligned}
        \underset{u,\,z^u,\,z^\ell,\,g_\mathcal{P}^u,\,g_\mathcal{P}^\ell}{\text{minimize}} \hskip 1em & \sum_{i=1}^{n_{g}} c_i (p_{\textrm{g,ref},i}+\alpha_i\Delta^u) \\
        \text{subject to} \hskip 1em & \text{\eqref{eqn:g_P},\;\eqref{eqn:conv_restr_eq}--\eqref{eqn:conv_restr_ineq2}} \; \text{and}\; \gamma=\gamma_\textrm{req}.
    \end{aligned}
    \label{eqn:robustOPF}
\end{equation}
Conveniently, the over-estimator of the distributed slack imbalance is given as an explicit optimization variable $\Delta^u$.
Replacing the constraints~\eqref{eqn:cost}--\eqref{eqn:OPconstr} with the robust feasibility condition~\eqref{eqn:conv_restr_eq}--\eqref{eqn:conv_restr_ineq2} from Theorem~\ref{thm:thm2} yields a tractable convex optimization problem whose solution is certifiably robustly feasible.
The conditions in \eqref{eqn:conv_restr_eq}--\eqref{eqn:conv_restr_ineq2} are constructed around a nominal dispatch point denoted as $u^{(k)}$.
The dependence on $u^{(k)}$ appears in~\eqref{eqn:J_eval}, where the Jacobian $J_{f,0}$ is evaluated at the given nominal point. 

Due to this dependence on the nominal operating point, we use an iterative solution approach.
We start by solving the nominal AC~OPF problem with any existing algorithm~\cite{ferc4} to determine the initial dispatch point $u^{(0)}$, which is not necessarily robustly feasible.
We improve this dispatch point via a \emph{sequential convex restriction} algorithm that iterates between (i) constructing~\eqref{eqn:robustOPF} by updating the convex restriction based on the dispatch point $u^{(k)}$ and (ii) solving \eqref{eqn:robustOPF} to obtain a new dispatch point $u^{(k+1)}$. 
The following steps describe our robust AC~OPF solution algorithm for some termination constants $\varepsilon_\textrm{nom}$ and $\varepsilon_\textrm{worst}$:

\noindent\textit{Initialization}: Solve the nominal AC~OPF problem~\eqref{eqn:cost}--\eqref{eqn:OPconstr} (i.e., $\mathcal{W}(0) = \left\lbrace w^{(0)}\right\rbrace$). Initialize $u^{(0)}$ and $x^{(0)}$ to the resulting solution. Set $\gamma=\gamma_\textrm{req}$ and $k=0$. The initial nominal cost $c^{(0)}$ and the worst-case operating cost $c^{u,(0)}$ are set to $\infty$ since the initial solution is generally not robustly feasible.

\begin{enumerate}[leftmargin=4em,labelindent=16pt,label=\bfseries Step \arabic*:]
  \item Set $x^{(k)}$ as the nominal operating point and evaluate Jacobians~\eqref{eqn:J_eval} to update the convex restriction. 
  \item Solve the robust AC OPF problem~\eqref{eqn:robustOPF} and update $u^{(k+1)}$ to the associated dispatch solution.
  \item Solve the AC power flow equations (e.g., using Newton's method) for the dispatch point $u^{(k+1)}$ with the nominal uncertain variable $w^{(0)}$ in order to obtain the corresponding nominal internal states $x^{(k+1)}$.
  \item If $c^{(k)}-c^{(k+1)}<\varepsilon_\textrm{nom}$ or $c^{u,(k)}-c^{u,(k+1)}<\varepsilon_\textrm{worst}$, return the dispatch point $u^{(k)}$ and terminate the algorithm. Otherwise, repeat from Step 1 with $k=k+1$.
\end{enumerate}

The dispatch point resulting from the proposed algorithm is guaranteed to be robustly feasible. By setting the initial condition as the solution to the nominal OPF problem, the output of the algorithm is often close to the nominal solution, i.e., the nominal solution provides our algorithm with a good initialization.
Step 4 requires the decrements of both the worst-case generation cost and the nominal generation cost to be less than the termination criteria. This requirement ensures that both the nominal and worst-case costs decrease at every iteration in the algorithm.

\subsection{Potential Infeasibility and Optimality Gap}
Any operating point output by the algorithm is guaranteed to be robustly feasible. However, the algorithm may not yield any output in some cases, and the solution is not guaranteed to be optimal.
To further explain various scenarios, we note:

\begin{enumerate}
    \item Problem \eqref{eqn:robustOPF} can be infeasible if the original problem does not admit a robustly feasible solution due to strict operational limits. A trivial example is a nominal OPF problem with no feasible solution.
    \item Problem \eqref{eqn:robustOPF} may be infeasible due to the conservativeness of the robust convex restriction. Since the robust convex restriction is an inner approximation, it may only cover a subset of the robustly feasible points.
    \item The algorithm does not guarantee the optimality of the solution. In other words, there may exist lower cost dispatch points that are robustly feasible. 
\end{enumerate}

The \emph{optimality gap} refers to the difference between the worst-case cost of our solution and the worst-case cost associated with the global optimum of the robust AC~OPF problem. We will use the optimality gap to quantify the conservatism associated with the robustness constraint and the robust convex restriction.
The following proposition provides a way to bound the optimality gap of our solution by comparing it to the nominal (non-robust) solution.

\begin{figure*}[!htbp]
    \centering
    \includegraphics[width=7.0in]{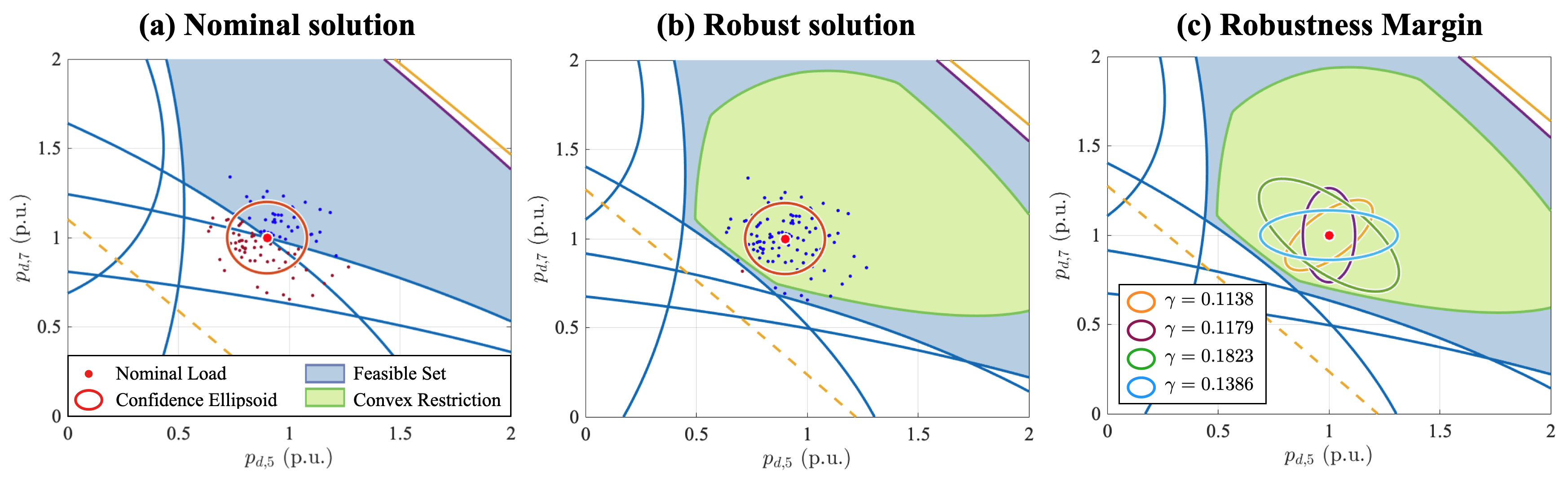}
    \caption{Fig.~\ref{fig:9bus}(a) shows the nominal (non-robust) AC~OPF solution for the 9-bus system from~\cite{chow1982time} with the feasible region (in blue) defined by uncertain load variation at buses~5 and~7.  Scattered dots show 100 load samples drawn from Gaussian distribution where blue and brown dots represent safe and unsafe realizations, respectively.
    Fig.~\ref{fig:9bus}(b) shows the robust AC~OPF solution obtained using the algorithm in Section~\ref{subsec:robust_acopf_algorithm} with a robustness margin equal to 20\% of the nominal load. The feasible space is shown in blue and the convex restriction is shown in green. Sampled uncertainty realizations that satisfy the operational constraints are shown by the blue points. Samples resulting in constraint violations are plotted in brown. Observe that all samples within the considered uncertainty set (red circle) are feasible. Fig.~\ref{fig:9bus}(c) shows the robustness margins $\gamma$ obtained by solving~\eqref{eqn:OPT_robust_feasibility} for various correlation matrices~$\Sigma$, with the dispatch point $u=u^{(0)}$ fixed to the solution in Fig.~\ref{fig:9bus}(b).}
    \label{fig:9bus}
    \vspace*{-0.5em}
\end{figure*}

\begin{proposition}
Suppose that $p_\textrm{g}^{(k)}$ is the active power dispatch obtained by solving problem \eqref{eqn:robustOPF} at iteration $k$.
Then, the optimality gap is bounded by
\begin{equation}
    \frac{c^u(p_\textrm{g}^{(k)}) - c^{u}(p_\textrm{g}^\textrm{robust})}{c^{u}(p_\textrm{g}^\textrm{robust})} \leq \frac{c^u(p_\textrm{g}^{(k)}) - c^{(0)}(p_\textrm{g}^\textrm{nom})}{c^{(0)}(p_\textrm{g}^\textrm{nom})},
\label{eqn:opt_gap}\end{equation}
where $p_\textrm{g}^\textrm{robust}$ and $p_\textrm{g}^\textrm{nom}$ refer to the active power dispatches for the robust OPF and nominal OPF solutions, respectively. The functions $c^u(p_\textrm{g}) = \max_{w\in\mathcal{W}(\gamma)}c(p_\textrm{g})$ and $c^{(0)}(p_\textrm{g}) = c(p_\textrm{g})\mid_{w=w^{(0)}}$ are the worst-case and nominal cost functions, respectively.
\label{prop:opt_gap}
\end{proposition}
\begin{proof}
Since $p_\textrm{g}^\textrm{nom}\in\argmin c^{(0)}(p_\textrm{g})$ subject to the OPF constraints~\eqref{eqn:ACPowerflow}--\eqref{eqn:OPconstr} and $c^{(0)}(p_\textrm{g})\leq c^u(p_\textrm{g})$ for all $p_\textrm{g}$,
\begin{displaymath}
     c^{(0)}(p_\textrm{g}^\textrm{nom})\leq c^{(0)}(p_\textrm{g}^\textrm{robust})\leq c^u(p_\textrm{g}^\textrm{robust}).
\end{displaymath}
Then, $c^{(0)}(p_\textrm{g}^\textrm{nom})\leq c^u(p_\textrm{g}^\textrm{robust})$, and \eqref{eqn:opt_gap} holds by rearranging the equations.
\end{proof}

The optimality gap is hard to determine precisely because we do not know the exact optimal robust solution, $p_\textrm{g}^\textrm{robust}$.
Proposition~\ref{prop:opt_gap} provides a way to numerically compute an upper bound on the optimality gap since the nominal optimal cost $c^{(0)}(p_\textrm{g}^\textrm{nom})$ is obtained from the nominal OPF solution.
The empirical studies in the next section show that the upper bound on the optimality gap is small ($\lesssim 1\%$), indicating that our robustly feasible dispatch points are at least close to optimal.

%\subsection{Deriving Uncertainty Set}
%Finally, we discuss the derivation of the uncertainty set in \eqref{eqn:ellipsoid}. The uncertainty set description in () can be converted from probability distribution.

\section{Numerical Studies}
\label{sec:experiments}
This section provides numerical studies to demtonstrate our algorithms. These numerical studies were performed on a computer with a 3.3~GHz Intel Core~i7 processor and 16~GB of RAM. Our implementations used JuMP/Julia~\cite{jump}. The convex quadratic problems were solved with MOSEK.
The initial OPF problems were solved using PowerModels.jl~\cite{powermodels} and IPOPT~\cite{wachter2006implementation}, and the AC power flow equations with distributed slack were solved using a customized implementation of Newton's method.
The numerical studies obtain dispatch points using our robust AC OPF solution, which guarantees robust feasibility. The implementation is available at \url{https://github.com/dclee131/PowerFlowCVXRS}.

We examine our dispatch solution in a stochastic setting where the uncertainty is drawn from Gaussian distributions. While the performance evaluation could use any probability distribution, the Gaussian distribution was chosen to match our assumption that the source of uncertainty is a short-term prediction and forecast error.

\subsection{Illustrative Example using a 9-Bus System}
We begin by considering the 9-bus system from~\cite{chow1982time} with uncertain loads at buses~5 and~7.
This 9-bus system is operating in a normal condition with positive active and reactive power load demands.
The participation factors are set to~1 for the generator at bus~1 and~0 for the other generators, which corresponds to the single slack bus formulation.

\subsubsection{Comparison of nominal and robust solution}
We first consider normally distributed loads with mean equal to the nominal load~$w^{(0)}$ and standard deviation equal to 10\% of~$w^{(0)}$ without correlation between loads.
For our robust AC~OPF algorithm, we use an uncertainty set $\mathcal{W}(\gamma_\textrm{req})$ that encloses two standard deviations of the considered load uncertainty by setting $\Sigma=\mathbf{diag}(p_\textrm{d}^2)$ and $\gamma_\textrm{req}=0.2$. 
To assess the quality of our solutions from the probabilistic aspect, we compute the probability of constraint violations using 10,000 samples of the random loads. 

Fig.~\ref{fig:9bus} shows the uncertainty set and the convex restriction for the loads at buses 5 and 7. The nominal loads and the confidence ellipsoid are intrinsic to the uncertainty model and do not change with the dispatch point.
Changes in the dispatch point adjust the set of feasible demands, and thus the feasible sets shown in Figs.~\ref{fig:9bus}(a) and (b) are different while the nominal load stays at the same location.
Specifically, the robust dispatch solution changes the feasible set of loading conditions such that it inscribes the confidence ellipsoid.
The quality of our solution is evident by examining the distance between the boundaries of the feasible region, convex restriction, and uncertainty set.
For the nominal (non-robust) OPF solution in Fig.~\ref{fig:9bus}(a), the operational cost is 5296.69~\$/hr and constraint violations occur in 55.18\% of the samples.
The most common violations are the maximum voltage magnitude limits at buses~6 and~8. 
For the robust OPF solution in Fig.~\ref{fig:9bus}(b), we observe that all uncertainty realizations within the considered uncertainty set (red circle) are feasible, as guaranteed by our algorithm. The generation cost for the robust dispatch point (5342.99~\$/hr) is 0.87\% greater than the cost for the nominal dispatch point, but only 0.13\% of the samples result in the constraint violations. 

\subsubsection{Robustness Margin Maximization} We next examine a scenario where the system operators want to determine the operating point with the largest robustness margin. First, we let the dispatch point $u$ be a decision variable and solve the problem in \eqref{eqn:OPT_robust_feasibility} that maximizes the robustness margin. The resulting robustness margin $\gamma$ was 0.380, which is 90\% greater than what was required ($\gamma_\textrm{req}=0.2$) in the robust OPF solution above.
This achieves robustness against $\pm38$\% fluctuations in loads at every bus while the generation cost (5383.96 \$/hr) only increased by 0.77\% relative to the robust OPF solution.

Second, we examine a scenario where the dispatch point is given by $u=u^{(0)}$ and introduce correlation between loads. We define the dispatch point $u^{(0)}$ as the solution in Fig.~\ref{fig:9bus}(b), and compute the robustness margin by solving~\eqref{eqn:OPT_robust_feasibility} with the additional constraint $u = u^{(0)}$. We solve this problem several times for different correlation matrices $\Sigma$.
The legend of Fig.~\ref{fig:9bus}(c) gives the robustness margins $\gamma$ obtained for different correlation matrices $\Sigma$, with the corresponding uncertainty sets shown by the ellipses.

\subsubsection{Computation Time}
Since the robustness condition is convex, these problems can be solved efficiently.
For the 9-bus system, the robust AC~OPF algorithm converged in two iterations, each taking an average of 0.0497~seconds to compute. 
The number of iterations indicates how many times the convex optimization problem in \eqref{eqn:robustOPF} was solved.
The average computation time for obtaining the robustness margins is 0.0585~seconds.

\begin{figure}[!htbp]
	\centering
	\includegraphics[width=3.2in]{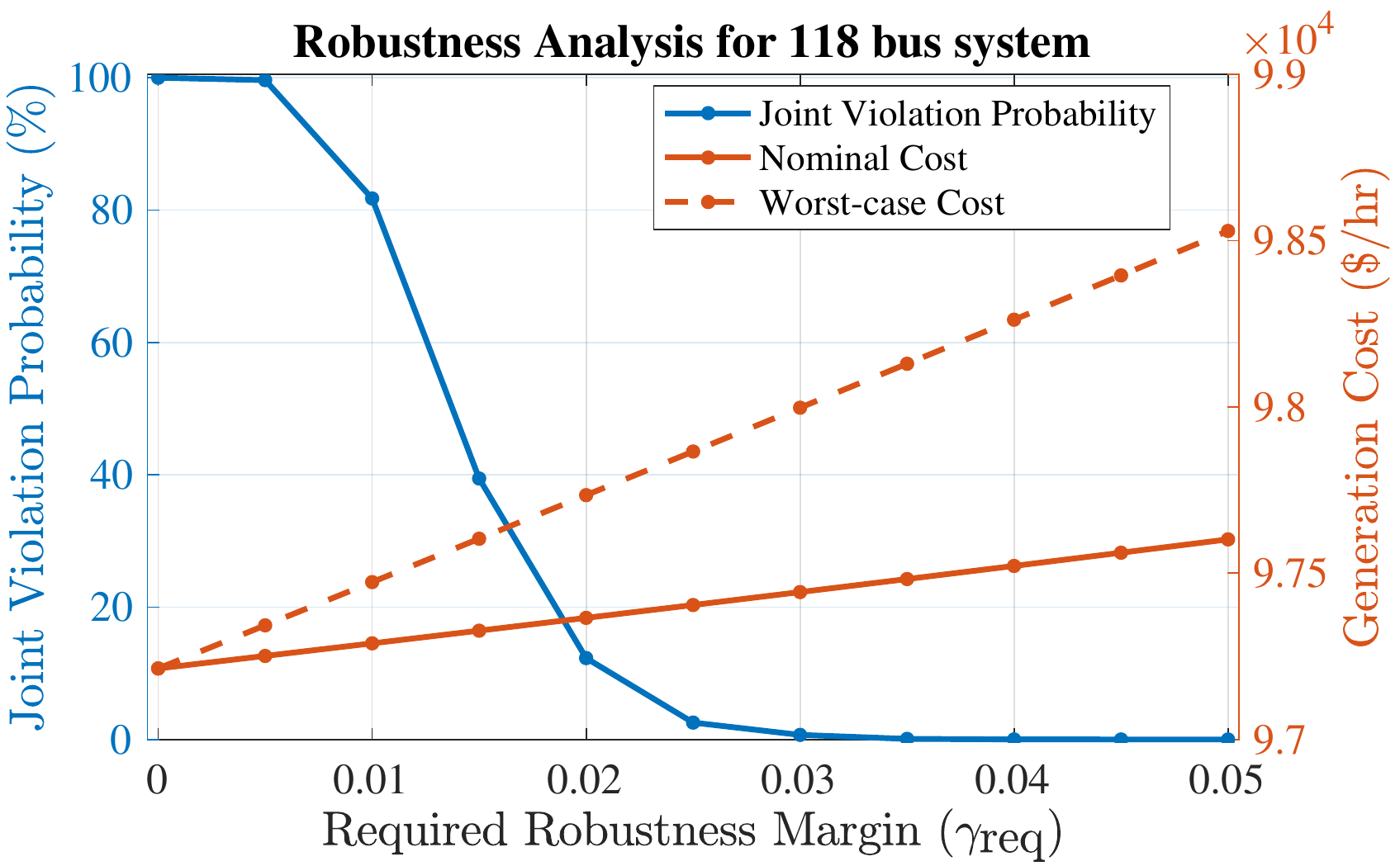}
	\caption{The trade-off between generation cost and robustness for the IEEE~118-bus system. %~\cite{dataset}
	The x-axis shows the required robustness margin,~$\gamma_\textrm{req}$. 
    The solid and dashed red lines show the nominal and worst-case generation costs, $c^{(0)}$ and $c^u$, respectively.
    The blue lines show the joint probability of constraint violations with stochastic uncertainty.}
	\label{fig:violation}
	%\vspace*{-0.5em}
\end{figure}

\subsection{Robustness vs. Cost Trade-Off for the IEEE~118-Bus System}
We next study the trade-off between operating cost and the probability of constraint violations by solving the robust AC~OPF problem for the IEEE 118-bus system~\cite{pglib}.
Load uncertainty is modeled via a Gaussian distribution with variance equal to $1\%$ of the nominal load.
The results are plotted in Fig.~\ref{fig:violation} with the enforced robustness margin $\gamma_\textrm{req}$ on the x-axis. The left y-axis (in blue) shows the empirically determined join violation probability, i.e., the probability that a sample from the considered uncertainty distribution violates one or more constraints. The right y-axis (in orange) shows both the nominal ($c^{(0)}$) and worst-case ($c^u$) generation costs.

We observe that increasing the robustness margin results in lower violation probability and higher generation cost. 
By examining this trade-off, operators can balance generation cost and robustness based on the assessed level of power injection uncertainty.
For example, we can avoid constraint violations with 99\% probability against random samples from the Gaussian distribution by setting $\gamma_\textrm{req}=0.03$, which increases the cost relative to the nominal solution by 0.24\%.

\begin{table*}[!htbp]
\centering
\caption{Violation Probability and Cost Comparison for Selected PGLib Test Cases}
\begin{tabular}{ |c||c|c|c|c|c|c|c|c|c|c|}
\hline
& \multicolumn{3}{c|}{Nominal OPF Solution}  &  \multicolumn{5}{c|}{Robust OPF Solution} & \multicolumn{2}{c|}{Optimality Gap} \\
& & \multicolumn{2}{c|}{Violation Probability (\%)}  & &  \multicolumn{2}{c|}{Violation Probability (\%)} & Solver Time & Num. of & \multicolumn{2}{c|}{(\%)} \\
Test Case & Cost (\$/h) & (Uniform) &  (Gaussian)  & Cost (\$/h) & (Uniform) &   (Gaussian) & per iter (sec) & iterations & Local & QC\! Relax \\ \hline \hline
\multicolumn{11}{|c|}{Typical Operating Conditions (TYP)} \\ \hline
case3\_lmbd & \hphantom{00}5812.64 & \hphantom{0}49.88 & \hphantom{0}49.47 & \hphantom{00}5829.93 & 0 & \hphantom{0}0.20 & 0.04& 4 & 0.30 & \hphantom{0}1.54 \\ \hline
case5\_pjm & \hphantom{0}17551.89 & 100.00 & 100.00 & \hphantom{0}17631.82 & 0 & \hphantom{0}5.18 & 0.07& 3 & 0.46 & 17.54 \\ \hline
case14\_ieee & \hphantom{00}2178.08 & \hphantom{0}91.30 & \hphantom{0}91.44 & \hphantom{00}2180.96 & 0 & \hphantom{0}2.10 & 0.12& 2 & 0.13 & \hphantom{0}0.54 \\ \hline
case24\_ieee\_rts & \hphantom{0}63352.20 & \hphantom{0}99.40 & \hphantom{0}99.66 & \hphantom{0}63566.97 & 0 & \hphantom{0}5.62 & 0.71& 3 & 0.34 & \hphantom{0}0.40 \\ \hline
case30\_as & \hphantom{000}803.13 & \hphantom{0}49.07 & \hphantom{0}50.10 & \hphantom{000}803.13 & 0 & \hphantom{0}1.94 & 0.26& 2 & 0.00 & \hphantom{0}0.06 \\ \hline
case30\_fsr & \hphantom{000}575.77 & \hphantom{0}69.12 & \hphantom{0}71.99 & \hphantom{000}580.01 & 0 & \hphantom{0}0.03 & 0.38& 2 & 0.74 & \hphantom{0}1.13 \\ \hline
case30\_ieee & \hphantom{00}8208.52 & \hphantom{0}68.74 & \hphantom{0}80.84 & \hphantom{00}8232.81 & 0 & \hphantom{0}2.70 & 0.38& 3 & 0.30 & 23.17 \\ \hline
case39\_epri & 138415.56 & \hphantom{0}99.96 & \hphantom{0}99.94 & 138643.93 & 0 & \hphantom{0}5.60 & 0.51& 4 & 0.16 & \hphantom{0}0.71 \\ \hline
case57\_ieee & \hphantom{0}37589.34 & \hphantom{0}96.91 & \hphantom{0}98.04 & \hphantom{0}37602.58 & 0 & \hphantom{0}5.34 & 0.95& 3 & 0.04 & \hphantom{0}0.22 \\ \hline
case73\_ieee\_rts & 189764.08 & 100.00 & 100.00 & 190139.37 & 0 & \hphantom{0}9.86 & 1.74& 3 & 0.20 & \hphantom{0}0.26 \\ \hline
case118\_ieee & \hphantom{0}97213.61 & 100.00 & 100.00 & \hphantom{0}97261.57 & 0 & 12.89 & 6.32& 3 & 0.05 & \hphantom{0}0.86 \\ \hline

\end{tabular}
\label{tab:result}
\vspace*{-1em}
\end{table*}

\subsection{Numerical Studies using the PGLib Test Cases}
Finally, we show the effectiveness of our robust AC~OPF algorithm using the PGLib test cases~\cite{pglib} with sizes up to 179 buses.
Each generator's participation factor is proportional to the generator's capacity, i.e., $\alpha_i=(p_{\textrm{g},i}^\textrm{max}-p_{\textrm{g},i}^\textrm{min})/\sum_{i=1}^{n_g} (p_{\textrm{g},i}^\textrm{max}-p_{\textrm{g},i}^\textrm{min})$.
The uncertainty set for the robust OPF problem was modeled by 1\% demand fluctuations at every load bus (i.e., $\Sigma=\mathbf{diag}(p_\textrm{d}^2)$ and $\gamma_\textrm{req}=0.01$).
Table~\ref{tab:result} compares the generation cost of the nominal (non-robust) solution obtained using PowerModels.jl~\cite{powermodels} with our robust dispatch point in the second and fifth columns, respectively. The bound on the optimality gap in Proposition \ref{prop:opt_gap} is shown in the tenth column, which is computed by taking the difference in generation costs between the nominal and robust OPF solution.
The optimality gap is approximately less than 1\%, indicating that only a marginal trade-off in generation cost is necessary to achieve this level of robustness. 
The robustness against stochastic uncertainty was evaluated by checking the feasibility of 10,000 samples from Gaussian and uniform distributions. Columns 3 and 6 in Table~\ref{tab:result} used samples from a uniform distribution within the uncertainty set. Since our robust OPF solution guarantees robust feasibility, there was no violated case for uniform distribution. Columns 4 and 7 used a Gaussian distribution with its mean set to the nominal loads and variance set to 0.5\% of the nominal loads. The results show that while the nominal OPF solution is very sensitive to the fluctuation, the robust solution makes the grid significantly more robust against stochastic uncertainty.
Columns 8 and 9 show the solver time and the number of iterations, respectively. All of the studied test cases required fewer than five iterations. Columns~10 and~11 show bounds on the optimality gaps computed by the right-hand side of~\eqref{eqn:opt_gap} in Proposition~\ref{prop:opt_gap}. The optimality gap in column 10 is calculated using the generation cost from a local search method (IPOPT~\cite{wachter2006implementation}), and column 11 uses the lower bound on the generation cost from the quadratic convex (QC) relaxation~\cite{coffrin2015qc}. We note that the optimality gap from the QC relaxation in column 11 includes both the relaxation gap and the gap from the robust convex restriction.

\section{Conclusion}
\label{sec:conclusion}

We presented a robust convex restriction condition and an optimization algorithm to solve robust AC optimal power flow problems. Our approach addresses inherent uncertainties in power injections due to the growing quantities of renewable generation. The main advantage of our approach is its theoretical guarantees with respect to both power flow solvability and operational constraint satisfaction for robust AC~OPF problems. Our numerical results demonstrate the proposed algorithm on a range of test cases, enabling studies of the trade-off between cost and robustness.

There are several open problems and extensions for future work. One direction is expanding the types of uncertainties considered in the formulation, such as impedance parameters and transmission line statuses. This includes \mbox{N-1} contingencies related to the failure of a single transmission line or generator. The \mbox{N-1} security criterion can be framed as a robust OPF problem with the line and generator statuses as uncertain variables. To handle this important type of uncertainties, future work will investigate methods for handling robustness constraints with discrete variables.

The proposed method considers a given uncertainty set in the form of a confidence ellipsoid. Other future work involves investigating the choice of the uncertainty set that best represents the underlying system. The required robustness margin, $\gamma_\textrm{req}$, needs to be chosen appropriately to mitigate the economic optimality and risk associated with the uncertainty. While we evaluate its probabilistic performance with numerical simulations, further investigation is needed to establish the relationship between the robustness margin and the probability of violating operational constraints.

Other future work involves extending the system model to consider additional control variables such as transformer tap ratios and phase shifts as well as generator participation factors. The power grid model can also be extended to consider three-phase networks. 
While the proposed method provides a general procedure, extensions to new models need to investigate appropriate nonlinear representations specific to the particular system description. 
The choice of nonlinear representation with a sparse coupling structure can play a critical role in the performance of robust convex restriction. Extended models will need to identify and exploit the system's structure to improve performance and scalability.

\appendix

\subsection{AC~OPF Constraints with Basis Functions} \label{appendix:ineq}
The reactive power injections at PV buses are defined as
\begin{equation}
C_\textrm{g,pv}\,q_\textrm{g}-C_\textrm{d,pv}\,q_\textrm{d}=\underbrace{\begin{bmatrix}
\mathbf{0} & -B^\textrm{c}_\textrm{pv} & G^\textrm{s}_\textrm{pv} & -B^\mathrm{d}_\textrm{pv}
\end{bmatrix}}_{L_\textrm{q}}\psi(z,u).
\label{eqn:pf_ineqn}
\end{equation}
Similarly, the transmission line flows are defined as
\begin{equation}
\begin{bmatrix} s_\textrm{p}^\textrm{f} \\ s_\textrm{q}^\textrm{f} \\ s_\textrm{p}^\textrm{t} \\ s_\textrm{q}^\textrm{t} \end{bmatrix}
=\underbrace{\begin{bmatrix}
\mathbf{0} & G_\mathrm{ft} & -B_\mathrm{ft} & G_\mathrm{ff}E_\mathrm{f}^T \\
\mathbf{0} & -B_\mathrm{ft} & -G_\mathrm{ft} & -B_\mathrm{ff}E_\mathrm{f}^T \\
\mathbf{0} & G_\mathrm{tf} & -B_\mathrm{tf} & G_\mathrm{tt}E_\mathrm{t}^T \\
\mathbf{0} & -B_\mathrm{tf} & -G_\mathrm{tf} & -B_\mathrm{tt}E_\mathrm{t}^T
\end{bmatrix}}_{[L_\textrm{p,line}^\textrm{f};\;L_\textrm{q,line}^\textrm{f};\;L_\textrm{p,line}^\textrm{t};\;L_\textrm{q,line}^\textrm{t}]} \psi(z,u).
\label{eqn:lineflow_t} \end{equation}
The submatrices are defined such that $s_\textrm{j}^\textrm{k}=L_\textrm{j,line}^\textrm{k}\psi(z,u)$ for $\textrm{k}\in\{\textrm{f},\textrm{t}\}$ and $\textrm{j}\in\{\textrm{p},\textrm{q}\}$.

\subsection{Proof of Lemma \ref{lemma:g_bound}}\label{appendix:g_bound}
We present a short proof the case for $g^{\mathrm{v},u}(v_k)$. Given that $x\in\mathcal{P}(z^u,z^\ell)$, the voltage magnitude at bus $k$ can be represented as $v_k=\alpha_i v_k^\ell + (1-\alpha_i)v_k^u$ where $\alpha\in[0,1]$. Since the function $g_k^u(x,u)$ is convex, 
\begin{equation}\begin{aligned}
g_l^{\mathrm{v},u}(v_k) &\leq \alpha g_l^{\mathrm{v},u}(v_k^\ell)+(1-\alpha) g_l^{\mathrm{v},u}(v_k^u) \\
& \leq \max{\{g_l^{\mathrm{v},u}(v_k^\ell), g_l^{\mathrm{v},u}(v_k^u)\}},
\end{aligned}\end{equation}
for all $v_k\in[v_k^\ell,v_k^u]$.
From the condition in Lemma \ref{lemma:g_bound}, $g_l^{\mathrm{v},u}(v_k^\ell) \leq g_{\mathcal{P},k}^{\mathrm{v},u}$ and $g_l^{\mathrm{v},u}(v_k^u) \leq g_{\mathcal{P},k}^{\mathrm{v},u}$. Therefore, $g_l^{\mathrm{v},u}(v_k)\leq g_l^{\mathrm{v},u}(v_k^\ell)$ for all $v_k\in[v_k^\ell,v_k^u]$. Proof for general case where the function $g_k(x,u)$ takes multiple arguments (e.g., $g_l^{\mathrm{cos},\ell}(v_l^\textrm{f},v_l^\textrm{t},\varphi_l)$) is available in~\cite{lee2018}.

\subsection{Proof of Lemma \ref{lemma:max_interval}}\label{appendix:max_interval}
Since $x_i\in[x_i^\ell,\,x_i^u]$ for all $i$, $c^Tx \leq (c^+)^T x^u + (c^-)^T x^\ell$. Define $x^*$ such that $x^*_i = x^u_i$ if $c_i\geq0$, and $x^*_i = x^\ell_i$ if $c_i<0$ for $i=1,\ldots,n$. Then, $c^Tx^* = (c^+)^T x^u + (c^-)^T x^\ell$ achieves the upper bound, and hence $x^*$ is the optimal solution.

\subsection{Proof of Lemma \ref{lemma:max_ellipsoid}}\label{appendix:max_ellipsoid}

The optimization problem can be written as $\max\{c^Tw\mid (w-w^{(0)})^T\Sigma^{-1}(w-w^{(0)})\leq\gamma)\}=\max\{c^T(w^{(0)}+\Sigma^{1/2}\tilde{w})\mid \lVert\tilde{w}\rVert_2\leq\gamma\}$ by a change of variables with $w=w^{(0)}+\Sigma^{1/2}\tilde{w}$. Since $\max\{c^T\hat{w}\mid \lVert\hat{w}\rVert_2\leq\gamma\} = \gamma\lVert c\rVert_2$, the optimal objective value is $c^Tw^{(0)}+\gamma\lVert c^T\Sigma^{1/2}\rVert_2$.

\subsection{Proof of Theorem \ref{thm_robustness}} \label{appendix:proof}
The condition in inequality \eqref{eqn:conv_restr_eq} ensures that, for $i=1,\ldots,n_\textrm{l}+n_\textrm{pq}+1$,
\begin{displaymath}
\begin{aligned}
     &\max_{w\in\mathcal{W}} \max_{x\in\mathcal{P}} \left\{K_ig(x,u)-A_iJ_{f,0}^{-1}Rw\right\} \\
     &\leq\max_{x\in\mathcal{P}} \left\{K_i^+g^u(x,u)+K_i^-g^\ell(x,u)\right\}-\min_{w\in\mathcal{W}}A_iJ_{f,0}^{-1}Rw \\
     &\leq K_i^+\max_{x\in\mathcal{P}}g^u(x,u)+K_i^-\min_{x\in\mathcal{P}}g^\ell(x,u)+\xi^u(\gamma) \\
     &=K_i^+g^u_\mathcal{P}+K_i^-g^\ell_\mathcal{P}+\xi(\gamma)\leq z_i^u.
\end{aligned}
\end{displaymath}
Similarly, $\min_{w\in\mathcal{W}} \min_{x\in\mathcal{P}} K_ig(x,u)-A_iJ_{f,0}^{-1}Rw\geq z_i^\ell$. These constraints ensure satisfaction of the inequalities in \eqref{eqn:self_mapping_P}.
Then the nonlinear map $G_{u,w}:x\rightarrow-J_{f,(0)}^{-1}Mg(x,u)-J_{f,(0)}^{-1}Rw$ in \eqref{eq:fixed-point-form} is self-mapping with the set $\mathcal{P}(z^u,\,z^\ell)$ for any realizations of $w\in\mathcal{W}(\gamma)$ (i.e., $\forall w\in\mathcal{W}(\gamma)$ and $\forall x\in\mathcal{P}(z^u,\,z^\ell)$, $G_{u,w}(x)\in\mathcal{P}(z^u,\,z^\ell)$).
By Brouwer's fixed-point theorem, there exists a fixed-point $x\in\mathcal{P}(z^u,\,z^\ell)$, which is equivalent to satisfying the AC power flow equations in~\eqref{eqn:pf_eqn}.

\subsection{Proof of Theorem \ref{thm:thm2}} \label{appendix:proof2}
From Theorem \ref{thm_robustness}, the conditions in \eqref{eqn:g_P} and \eqref{eqn:conv_restr_eq} ensure that there is an internal state in the set $\mathcal{P}(z^u,\,z^\ell)$ for all realizations of the uncertainty set $\mathcal{W}(\gamma)$.
The inequalities in~\eqref{thm:operational_constr}--\eqref{eqn:conv_restr_ineq2} ensure that $x$ satisfies the operational constraints in~\eqref{eqn:OPconstr} for all $x\in\mathcal{P}(z^u,\,z^\ell)$. The inequalities in \eqref{eqn:conv_restr_ineq1} ensure that
\begin{displaymath} \begin{aligned}
     &\max_{x\in\mathcal{P}} L_{\textrm{q},i}\psi(x,u)\leq L_{\textrm{q},i}^+\max_{x\in\mathcal{P}} \psi^u(x,u) + L_{\textrm{q},i}^-\min_{x\in\mathcal{P}} \psi^\ell(x,u) \\
     & \leq L_{\textrm{q},i}^+\psi^u_\mathcal{P}+L_{\textrm{q},i}^-\psi^\ell_\mathcal{P} \leq C_{\textrm{g,pv},i}\, q_g^u - \zeta^u_i(\gamma) \\
     &\leq C_{\textrm{g,pv},i}\, q_g^u - \max_{w\in\mathcal{W}}C_{\textrm{d,pv},i}\Sigma^{1/2}_\textrm{q}w \\
     & \leq \min_{w\in\mathcal{W}} \left(C_{\textrm{g,pv},i}\, q_g^u - C_{\textrm{d,pv},i}\, \Sigma^{1/2}_ \textrm{q}w\right),
\end{aligned} \end{displaymath}
and thus, $L_{\textrm{q},i}\psi(x,u)\leq C_{\textrm{g,pv},i}\,q_g^u - C_{\textrm{d,pv},i}\,q_\textrm{d}$ holds for all $x\in\mathcal{P}(z^u,\,z^\ell)$ and $w\in\mathcal{W}(\gamma)$. Therefore, $q_\textrm{g}^u$ is an upper bound on the reactive power generation for all uncertainty realizations $w\in\mathcal{W}(\gamma)$.
Similarly, the inequalities in \eqref{eqn:conv_restr_ineq1} and \eqref{eqn:conv_restr_ineq2} ensure that $q_\textrm{g}^\ell$ is a valid lower bound on the reactive power generation, and $s_\textrm{p}^{\textrm{t},u}$, $s_\textrm{q}^{\textrm{f},u}$, and $s_\textrm{q}^{\textrm{t},u}$ are valid line flow bounds.
Therefore, for each uncertainty realization $w\in\mathcal{W}(\gamma)$, there exist associated internal states $x$ that satisfy the power flow equations and the operational constraints.

\ifCLASSOPTIONcaptionsoff
\newpage
\fi

\IEEEtriggeratref{19}
\bibliographystyle{IEEEtran}
\bibliography{references}
%\nocite{*}

\end{document}